\documentclass[journal]{IEEEtran}

\usepackage{amsfonts}
\usepackage{subfigure}
\usepackage{amsmath}
\usepackage{mathtools}
\allowdisplaybreaks[4]
\usepackage{amsthm}
\usepackage{amssymb}
\usepackage{mathrsfs}
\usepackage{bm}
\usepackage{float}
\usepackage{dsfont}
\usepackage{cite}
\usepackage{multicol}
\usepackage{multirow}
\usepackage{tabu}
\usepackage{booktabs}
\usepackage{algorithmic,algorithm}
\usepackage{lipsum}
\newtheorem{defi}{Definition}

\newtheorem{assum}{Assumption}

\newtheorem{lem}{Lemma}
\newtheorem{theo}{Theorem}
\newtheorem{coro}{Corollary}
\newtheorem{exmp}{Example}
\DeclareMathOperator*{\minimize}{minimize}

\DeclareMathOperator*{\st}{subject~to}
\DeclareMathOperator*{\sta}{s.t.}

\begin{document}
%
%

%
%

\title{Second-order Conic Programming Approach for Wasserstein Distributionally  Robust Two-stage  Linear Programs}

%
%
%

\author{Zhuolin Wang,~Keyou You,~Shiji Song
        and~Yuli Zhang
\thanks{This work was supported by the National Natural Science Foundation of China [grant numbers 61722308, U1660202 and 71871023].}
\thanks{Zhuolin Wang, Keyou You and Shiji Song are with the Department of Automation and BNRist, Tsinghua University, Beijing 100084, China  (e-mail: wangzl17@mails.tsinghua.edu.cn, youky@tsinghua.edu.cn, shijis@tsinghua.edu.cn.)}
\thanks{Yuli Zhang is with School of Management and Economics, Beijing Institute of Technology, Beijing 100081, China. (e-mail:zhangyuli@bit.edu.cn)}}


\markboth{IEEE TRANSACTIONS ON AUTOMATION SCIENCE AND ENGINEERING}
{Shell \MakeLowercase{\textit{et al.}}: Bare Demo of IEEEtran.cls for Journals}

\maketitle

\begin{abstract}
This paper proposes a second-order conic programming (SOCP) approach to solve distributionally robust two-stage stochastic linear  programs over 1-Wasserstein balls. We start from the case with distribution uncertainty only in the objective function and  {\em exactly} reformulate it as an SOCP problem. Then, we study the case with distribution uncertainty only in constraints, and show that such a robust program is generally NP-hard as it involves a norm maximization problem over a polyhedron. However, it is reduced to an SOCP problem if the extreme points of the polyhedron are given as a prior.  This motivates to design a constraint generation algorithm with provable convergence to approximately solve the NP-hard problem. In sharp contrast to the exiting literature,  the distribution achieving the worst-case cost is given as an ``empirical" distribution by simply perturbing each sample for both cases. Finally, experiments illustrate the advantages of the proposed model in terms of the out-of-sample performance and the computational complexity.
\end{abstract}

\def\abstractname{Note to Practitioners}
\begin{abstract}
 The two-stage program with the distribution uncertainty is a fundamental optimization problem with broad applications. This paper proposes a novel data-driven model over 1-Wasserstein balls to handle the distribution uncertainty. Moreover, it applies to the situation where the true distribution is slowly time-varying. An SOCP-based algorithm with provable convergence has been presented to solve the two-stage program over 1-Wasserstein balls. The proposed model is particularly effective for applications with uncertainty only in the objective functions or constraints. This paper focuses on the applications in two-stage portfolio programs and two-stage material order programs. The good out-of-sample performance and computational complexity of the model are validated by the experiments on these applications. This model can be easily applied to other applications, such as the two-stage schedule problems and recourse allocation problems.

\end{abstract}



\IEEEpeerreviewmaketitle


\section{Introduction}
The two-stage program is one of the most fundamental optimization problems and has broad applications, see e.g., \cite{zhang2018ambulance,9046836}. It is observed that its coefficients are usually uncertain and ignoring their uncertainties may lead to poor decisions \cite{paridari2015robust,9093973}.
In the literature, the classical robust optimization (RO) has been proposed to handle the uncertainty in the two-stage program by restricting them to some given sets and then minimizes the worst-case cost over all possible realizations \cite{ben2009robust}. However, it ignores the distribution information of stochastic uncertainty and may return a conservative solution \cite{zhao2018inventory}. To this end, the stochastic program (SP) is adopted to address the uncertainty via a distribution function \cite{shapiro2009lectures}, and in practice is solved by using an empirical distribution in the sample-average approximation (SAA) method \cite{shapiro1998simulation}. The SAA method is effective only when adequate and high-quality samples are obtained cheaply \cite{shapiro1998simulation}. If samples are of low quality, the empirical distribution may significantly deviate from the true distribution and the SAA method exhibits poor performance.

An alternative approach is to apply the distributionally robust (DR) optimization technique to address stochastic uncertainty  by assuming that the true distribution belongs to an ambiguity set of probability distributions \cite{shapiro2002minimax,shi2018distributionally}. This method overcomes inherent drawbacks of the SP and RO as it does not require an exact distribution and can exploit the sample information. In fact, numerous evidence implies that the DR method can yield high-quality solutions within a reasonable computation cost \cite{van2015distributionally,Zhang2017Robust,xiong2016distributionally}. Thus, our exposition concentrates on DR two-stage linear programs over an ambiguity set of distributions.

The ambiguity set is essential in the DR programs. It should be large enough to include the true distribution with a high probability but cannot be too ``large" to avoid very conservative decisions. \cite{bertsimas2010models,hanasusanto2016k,ling2017robust} adopt the moment-based ambiguity set, which includes distributions with specified moment constraints.
 The DR two-stage linear program over the set of distributions with {\em{exactly}} known  first- and second-order moments are reformulated either as a semidefinite program \cite{bertsimas2010models} or the mixed-integer linear program of a polynomial size \cite{hanasusanto2016k} under different settings. Observe that the moment mismatch is unavoidable, \cite{ling2017robust} further considers the moment uncertainty, which results in an intractable model.

In this work, we study a data-driven DR two-stage linear program over a ball centered at the empirical distribution of a finite sample dataset, and the ball radius reflects our confidence in the empirical distribution. Particularly, the lower the confidence, the larger the ball radius.
The sample dataset can be utilized in a flexible way to handle the distribution uncertainty, e.g., the degree of conservatism can be controlled by tuning the radius. Moreover, our model applies to the situation where the true distribution is slowly time-varying.


Note that the empirical distribution is discrete and the true distribution is usually continuous. We adopt the 1-Wasserstein metric to measure the distance between distributions,  which is different from the Kullback-Leibler divergence in \cite{chen2018distributionally} and $L^1$-norm in \cite{jiang2018risk}. Then, we obtain the DR two-stage linear program over 1-Wasserstein balls and develop a second-order conic programming (SOCP) approach to solve it. Since the Wasserstein ball contains the true distribution with a high probability \cite{esfahani2018data}, the proposed DR problem is expected to exhibit good out-of-sample performance. Moreover, the Wasserstein ball can asymptotically  degenerate to the true distribution as the sample size increases to infinity \cite{esfahani2018data}.

This work considers the distribution uncertainty either in the objective function or constraints of two-stage linear programs. Specifically, we first study the case with distribution uncertainty only in the objective function and  {\em exactly} reformulate it as an SOCP problem, which covers all the results of the conference version of this work \cite{wang2020solving}. Then we proceed to the case with the distribution uncertainty only in constraints and show that such a program is generally NP-hard as it requires to solve a norm maximization problem over a polyhedron.  The good news is that the resulting program can be reduced to an SOCP problem if the extreme points of the polyhedron are given as a prior. Motivated by this and also inspired by \cite{zeng2013solving,ding2015parallel}, we design a novel constraint generation algorithm with provable convergence to approximately solve it.

It should be noted that \cite{hanasusanto2018conic} and \cite{xie2019tractable} study the DR two-stage linear programs with the $2$-Wasserstein and $\infty$-Wasserstein metrics, respectively. In \cite{hanasusanto2018conic}, the distribution uncertainty arises simultaneously in the objective function and
constraints, which renders their model NP-hard, and the co-positive programs are utilized to approximately solve it.  \cite{xie2019tractable} reformulates the DR model as a computational demanding mixed-integer problem.  In comparison, we \textit{exactly} reformulate our model with distribution uncertainty only in the objective as an SOCP problem and design an SOCP approach to approximately solve the NP-hard problem with uncertainty only in constraints. Moreover, we explicitly derive the distribution achieving the worst-case cost by simply perturbing each sample, based on which we can further assess the quality of an optimal decision. This is clearly in sharp contrast  to  \cite{bertsimas2018adaptive}, \cite{hanasusanto2018conic} and \cite{xie2019tractable}. Overall, we summarized our contributions as follows:
\begin{itemize}
	\item We propose a novel SOCP approach to solve the data-driven DR two-stage linear programs over 1-Wasserstein balls.
	\item We \textit{exactly} reformulate the model with uncertainty only in the objective as a solvable SOCP problem.
	\item The model with uncertainty only in the constraints is shown to be NP-hard. To approximately solve it, we develop an SOCP-based constraint generation algorithm with provable convergence.
	\item The good out-of-sample performance and the computational complexity of our model are validated by experiments.
\end{itemize}
The rest of this paper is organized as follows. Section \ref{Pro-For} proposes the DR two-stage linear program over the 1-Wasserstein ball. Section \ref{uino}  reformulates the model with the distribution uncertainty only in the objective function as a tractable SOCP problem. Section \ref{uinc} studies the model with uncertainty only in constraints and presents an SOCP-based constraint generation algorithm. Section \ref{wos-cas} derives the distribution achieving the worst-case cost. Section \ref{sim} reports numerical results to illustrate the performance of the proposed model and the paper is concluded in Section \ref{con}. \\

{\bf Notation}: We denote the set of real positive real  numbers by $\mathbb{R}$ and $\mathbb{R}_+$. The boldface lowercase letter denotes a vector, e.g., $\bm{x} = [x_1,\dots,x_n]^T\in \mathbb{R}^n$. Special vectors include the zero vector $\bm{0}$ and the all one vector $\bm{e}$. $\|\cdot\|_p$ denotes the $l_p$-norm. Let $[N]=\{1,2,\dots,N\}$ and $|\mathcal{E}|$ denotes the cardinality of $\mathcal{E}$. The letters s.t. are an abbreviation of the phrase ``subject to ". $\text{Diag}(\cdot)$ denotes a diagonal matrix with vector $(\cdot)$ being diagonal elements.
\section{Problem Formulation} \label{Pro-For}
\subsection{The Two-stage Stochastic Linear Optimization}\label{cla-pro}
Consider the classical two-stage stochastic linear program  \cite{birge2011introduction}
\begin{equation}
	\begin{aligned}
		\label{classical}
		\minimize_{\bm{x} \in \mathcal{X}} \ \ \bm{c}^T\bm{x}  + \mathbb{E}_\mathbb{F}[Q(\bm{x},\bm{\xi})],
	\end{aligned}
\end{equation}
where $\bm{x} \in \mathbb{R}^{n}$ is the first-stage decision vector from a compact set $\mathcal{X}$ and is decided before the realization of a random vector $\bm{\xi}\in \mathbb{R}^{m}$ with the distribution $\mathbb{F}$.

The second-stage cost is evaluated based on the expectation of the following recourse problem
\begin{equation}
\begin{aligned}
\label{LO-Two}
Q(\bm{x},\bm{\xi}) = &\min \ \bm{z}(\bm{\xi})^T\bm{y} \\
&\begin{array}{ll}
 \sta & A(\bm{\xi})\bm{x}+B\bm{y} \ge \bm{b}(\bm{\xi})\\
&\bm{y} \in \mathbb{R}^{m}_+,
\end{array}
\end{aligned}
\end{equation}
where $B\in \mathbb{R}^{k\times m}$ is the \textit{recourse matrix} and $\bm{z}(\bm{\xi}) \in \mathbb{R}^{m},A(\bm{\xi}) \in \mathbb{R}^{k\times n}$ and $\bm{b}(\bm{\xi}) \in\mathbb{R}^{k}$ depend on the random vector $\bm{\xi}$.

In the sequel, we study models with uncertainty only in the objective function or constraints, each of which is motivated by two notable examples, see also \cite{ling2017robust,bertsimas2010models,bertsimas2018adaptive}.\\

\begin{exmp}(\cite{ling2017robust})
	\label{exmp-f}
	Consider a portfolio program with $n$ assets which investors can invest in two stages.  Generally the return for assets in the second stage is random, hence a stochastic two-stage portfolio program is designed for a maximum return 	\begin{equation}
	\begin{aligned}
	\label{port}
	\minimize_{\bm{e}^T\bm{x} = 1,~\bm{x} \ge \bm{0}} \ \ -(\bm{e}+\bm{c})^T\bm{x} + \mathbb{E}_{\mathbb{F}}[Q(\bm{x,\bm{\xi}})],
	\end{aligned}
	\end{equation}
	where $\bm{x},\bm{c}\in \mathbb{R}^n$ are vectors of the invested dollar and the return for the $n$ assets in the first stage, $Q(\bm{x,\xi})$ is given by
	\begin{align}
		\label{port_q}
		Q(\bm{x,\xi})=  \min  \ \ & -(\bm{e}+\bm{\xi})^T\bm{y}  \\
		{\rm s.t.} \ \ & \bm{y} \ge \bm{0}, \bm{\Delta}^s \ge 0, \bm{\Delta}^b \ge 0 \nonumber \\
		& A\bm{x}+(1-\theta)\bm{\Delta}^b-(1+\theta)\bm{\Delta}^s = \bm{y}, \nonumber
	\end{align}
	where $\bm{y},\bm{\xi}\in \mathbb{R}^n$ are vectors of the invested dollar and the random return for the assets in the second stage.  The matrix $A=\text{Diag}(\bm{e}+\bm{c})$, $\bm{\Delta}^s$ and $\bm{\Delta}^b$ are the vectors of the dollar for selling and buying the assets, and $\theta$ is the transaction cost.\\
\end{exmp}

\begin{exmp}(\cite{kall1994stochastic})
	\label{exmp-c}
	Consider a material order problem with $n$ raw materials and $m$ desired products. Let $\bm{b}\in \mathbb{R}^m$ denote the market demand vector for products. Let $a_{ij}$ be the amount of product $i$ produced by per unit of material $j$ and $A = [a_{ij}]_{m\times n}$ be the matrix of the production amount for all materials.

The market demand is usually time-varying and the uncertainty in the production amount is generally inevitable due to the quality of raw materials. Hence, it is unavoidable to introduce uncertainty $\bm{\xi}$ to the demand vector $\bm{b}$ and the matrix $A$, then the order problem is formulated as
	\begin{align}
		\label{order}
		&\minimize_{\bm{e}^T\bm{x} \le u, ~ \bm{x} \ge \bm{0}} \left\{\bm{c}^T\bm{x} + \mathbb{E}_\mathbb{F}[Q(\bm{x},\bm{\xi})]\right\},
	\end{align}
	where $u$ is the capacity of $n$ materials, $\bm{c}\in \mathbb{R}^n$ is the cost vector of $n$ materials, and $Q(\bm{x},\bm{\xi})$ is given as
		\begin{equation}
		\begin{aligned}
		\label{order-1}
		Q(\bm{x},\bm{\xi}) = & \min \ \ \bm{z}^T\bm{y} \\
		&\begin{array}{ll}
		{\rm s.t.}  &A(\bm{\xi})\bm{x}+\bm{y} \ge \bm{b}(\bm{\xi})\\
		&\bm{y} \in \mathbb{R}^{m}_+,
		\end{array}
		\end{aligned}
		\end{equation}
	where $\bm{z}\in\mathbb{R}^m$ is the penalty vector for per unit of undeliverable products and $\bm{y}\in\mathbb{R}_+^m$ is the corresponding shortage amount vector.
\end{exmp}
%
	

Motivated by above examples, we consider that $\bm{z}(\bm{\xi})$, $A(\bm{\xi})$ and $\bm{b}(\bm{\xi})$ in \eqref{classical}  depend affinely  on $\bm{\xi}$, i.e.,
\begin{equation}
\begin{aligned}
\label{uncertainc}
&\bm{z}(\bm{\xi}) = \bm{z}_0 + \sum\limits_{i=1}^{m}\xi_i\bm{z}_i, ~
 A(\bm{\xi}) = A_0 + \sum\limits_{i=1}^{m}\xi_iA_i, \\
  &\bm{b}(\bm{\xi}) = \bm{b}_0 + \sum\limits_{i=1}^{m}\xi_i\bm{b}_i,
\end{aligned}
\end{equation}
where  $\bm{z}_0,\bm{z}_1,\dots,\bm{z}_{m} \in \mathbb{R}^{m}$, $\bm{b}_0,\bm{b}_1,\dots,\bm{b}_{m} \in \mathbb{R}^{k}$ and  $A_0,A_1,\dots,A_{m} \in \mathbb{R}^{k \times n}$ are given as prior. In fact, the affine uncertainty has also been adopt in \cite{bertsimas2018adaptive,ling2017robust}.

The following condition guarantees the feasibility of the second-stage problem in \eqref{LO-Two} and is satisfied by many problems, e.g., the production planning problem, the newsvendor problem and its variants \cite{birge2011introduction}. \\

\begin{assum}
	\label{rel-com}
	The second-stage problem in \eqref{LO-Two} is always feasible for any $\bm{x} \in \mathcal{X}$ and $\bm{\xi}$.
\end{assum}

\subsection{Distributionally Robust Two-stage  Problems} \label{TDROLOP}
The program in \eqref{classical} generally requires an exact distribution $\mathbb{F}$ of $\bm{\xi}$. In practice, $\mathbb{F}$ can only be estimated through a finite sample dataset $\{\widehat{\bm{\xi}}^i\}_{i=1}^N$ and a common idea is to adopt the SAA method, where $\mathbb{F}$ is approximated by an empirical distribution $\mathbb{F}_N$ over the sample dataset, i.e., $$\mathbb{F}_N(\bm{\xi})=\frac{1}{N}\sum_{i=1}^{N}\bm{1}_{\{\widehat{\bm{\xi}}^i\le\bm{\xi}\}},$$
where $\bm{1}_A$ is the indicator of event $A$. Then the stochastic linear problem in \eqref{classical} is approximated by
\begin{equation}
\label{SSP-two}
\minimize_{\bm{x} \in \mathcal{X}} \ \ \left\{\bm{c}^T\bm{x}  + \frac{1}{N}\sum_{i=1}^{N}Q(\bm{x},\widehat{\bm{\xi}}^i)\right\}.
\end{equation}

By Glivenko-Cantelli theorem \cite{cantelli1933sulla}, the distribution $\mathbb{F}_N$ weakly converges to the true distribution $\mathbb{F}$ as $N$ increases to infinity. This implies the asymptotic convergence of \eqref{SSP-two} to the stochastic model \eqref{classical}. Hence, the SAA method is sensible only when $\mathbb{F}_N$ well approximates the true distribution $\mathbb{F}$.

However, insufficient and/or low-quality samples may lead to an empirical distribution $\mathbb{F}_N$ far from the true distribution $\mathbb{F}$. Thus, the SAA model \eqref{SSP-two} may be not reliable with poor out-of-sample performance.

As in \cite{esfahani2018data}, a data-driven approach is adopted to address the distribution uncertainty in this work.  We assume that $\mathbb{F}$ belongs to an ambiguity set $\mathcal{F}_N$ including all distributions within $\epsilon_N$-distance from the empirical distribution $\mathbb{F}_N$. Here $\epsilon_N$ indicates the confidence on $\mathbb{F}_N$, e.g., the larger the $\epsilon_N$, the lower the confidence.

Since the true distribution $\mathbb{F}$ is generally continuous and the empirical distribution $\mathbb{F}_N$ is discrete, the 1-Wasserstein metric \cite{ambrosio2013user} is adopted to measure their distance and consequently a $1$-Wasserstein ball $\mathcal{F}_N$ is obtained. Then we are interested in the worst-case second-stage cost over $\mathcal{F}_N$, i.e.,
\begin{equation}
\begin{aligned}
\label{ADRO-Two}
\beta(\bm{x})=\sup_{\mathbb{F}\in \mathcal{F}_N} \mathbb{E}_{\mathbb{F}}[Q(\bm{x},\bm{\xi})],
\end{aligned}
\end{equation}
and the DR two-stage linear program is formulated as
\begin{equation}
\begin{aligned}
\label{primal}
\minimize_{\bm{x} \in \mathcal{X}} \ \ \bm{c}^T\bm{x}  + \beta(\bm{x}).
\end{aligned}
\end{equation}
To evaluate an optimal solution, we also derive the worst-case distribution $\mathbb{F}^*$ that achieves the worst-case second-stage cost, i.e.,
\begin{equation}
\label{worst-dis}
\beta(\bm{x})=\sup_{\mathbb{F}\in \mathcal{F}_N} \mathbb{E}_{\mathbb{F}}[Q(\bm{x},\bm{\xi})] = \mathbb{E}_{\mathbb{F}^*}[Q(\bm{x},\bm{\xi})].
\end{equation}

\subsection{Ambiguity Set via the $1$-Wasserstein Metric}
We introduce the $r$-Wasserstein metric below.

\begin{defi}(\cite{ambrosio2013user}) Let $d(\bm{\xi}^1,\bm{\xi}^2)=\|\bm{\xi}^1-\bm{\xi}^2\|_p$ be the $l_p$-norm of $\bm{\xi}^1-\bm{\xi}^2$ on $\mathbb{R}^n$ and $(\Xi,d)$ be a Polish metric space.
	Given a pair of distributions $\mathbb{F}_1\in \mathcal{M}(\Xi)$ and $\mathbb{F}_2\in \mathcal{M}(\Xi)$ where $\mathcal{M}(\Xi)$ is a set containing all distributions supported on $\Xi$, the r-Wasserstein metric $W^r$:$\mathcal{M}(\Xi) \times \mathcal{M}(\Xi) \rightarrow \mathbb{R}_+$ is defined as
	\begin{equation}\label{defdis}
		\begin{aligned}
			& W^r(\mathbb{F}_1,\mathbb{F}_2)= \inf\left\{\left(\int_{\Xi^2} d(\bm{\xi}^1,\bm{\xi}^2)^r K(\mathrm{d}\bm{\xi}^1,\mathrm{d}\bm{\xi}^2)\right)^{1/r}\right.:\\ &\left.\int_{\Xi} K(\bm{\xi}^1,\mathrm{d}\bm{\xi}^2)=\mathbb{F}_1(\bm{\xi}^1),
			\int_{\Xi} K(\mathrm{d}\bm{\xi}^1,\bm{\xi}^2)=\mathbb{F}_2(\bm{\xi}^2) \right\},
		\end{aligned}
	\end{equation}
	where $r\ge 1$ and $K$ is a joint distribution with its marginal distributions being $\mathbb{F}_1$ and $\mathbb{F}_2$.
\end{defi}

Without scarifying much modeling power and to obtain a real  metric \cite{ambrosio2013user}, we need the following requirement on the set $\mathcal{M}(\Xi)$.

\begin{assum}
	\label{M-set-assum}
	For any distribution $\mathbb{F} \in \mathcal{M}(\Xi)$, it holds
	$$\int_{\Xi}{\| {\bm{\xi}}\|}^r_p \mathbb{F}(\mathrm{d}{\bm{\xi}})<\infty.$$
\end{assum}
Different from \cite{hanasusanto2018conic} and \cite{xie2019tractable}, we adopt the 1-Wasserstein metric and $l_2$-norm, i.e., $r=1$ and $p=2$ in \eqref{defdis} to construct the ambiguity ball $\mathcal{F}_N$,
\begin{equation}
	\label{Wset}
	\mathcal{F}_N=\{\mathbb{F}\in \mathcal{M}(\Xi):W^1(\mathbb{F}_N,\mathbb{F})\le \epsilon_N\},
\end{equation}
where $\epsilon_N > 0$ is the ball radius, i.e., $\mathcal{F}_N$ is the set of distributions within $\epsilon_N$-distance from $\mathbb{F}_N$.

\subsection{Comparisons with the state-of-the-art methods}

In \cite{bertsimas2018adaptive}, the ambiguity set of the DR two-stage linear programs is defined as a set of distributions with specified first- and second-order moment constraints.

\cite{hanasusanto2018conic} considers DR two-stage linear programs of the form \eqref{primal}  with 2-Wasserstein balls, i.e., $r = p=2$ in \eqref{defdis}, and $Q(\bm{x},\bm{\xi})$ is defined as
\begin{equation}
\begin{aligned}
\label{Kuhn}
	\begin{array}{ll}
	Q(\bm{x},\bm{\xi}) =
	&\min \ \  (Q\bm{\xi}+\bm{q})^T\bm{y} \\
	&\sta \ \   T(\bm{x})\bm{\xi}+h(\bm{x})\le B\bm{y}
	\end{array}
\end{aligned}
\end{equation}
where $T(\cdot)$ and $h(\cdot)$ are two affine functions.

In \cite{xie2019tractable}, the DR two-stage program is defined via the $\infty$-Wasserstein metric, i.e, $r = \infty$ and $p=1, \infty$ in \eqref{defdis} with the uncertainty only in the objective function or constraints separately, i.e., $Q$ or $T(\bm{x})$ in \eqref{Kuhn} is set to $0$ respectively.


Comparisons with those state-of-art models are summarized as follows:

\begin{itemize}
	\item \textbf{Model differences:}
	Clearly, $Q(\bm{x},\bm{\xi})$ in \eqref{LO-Two} of this work and \cite{bertsimas2018adaptive} is different from \eqref{Kuhn} in \cite{hanasusanto2018conic} and \cite{xie2019tractable}.  Our model is motivated from a wide range of real applications, see e.g. Examples \ref{exmp-f}-\ref{exmp-c}. Note that this ``minor" difference may require a completely different solution approach.
	\item \textbf{Solution approaches:}  \cite{hanasusanto2018conic} derives  co-positive programs to approximate their NP-hard DR two-stage model. \cite{xie2019tractable} reformulates the model as a computational demanding mixed-integer problem. \cite{bertsimas2018adaptive} approximate their model by linear decision rule techniques.
	
	In this work, we {\em{equivalently}} reformulate our model with distribution uncertainty only in the objective as an SOCP problem and design an SOCP-based constraint generation algorithm for the problem with distribution uncertainty only in constraints.

	\item \textbf{Approximation gaps:} There is no approximation gap in \cite{hanasusanto2018conic} and \cite{bertsimas2018adaptive}, under the condition that for any $\bm{t} \in \mathbb{R}^k$, there exists a solution $\bm{y}$ to solve the inequality $B\bm{y} \ge \bm{t}$ (aka {\em complete recourse}). In this work,  the zero-gap condition in Assumption \ref{rel-com} (aka {\em relatively complete recourse}) is weaker and satisfied by numerous real application models \cite{birge2011introduction}.
	
	As explicitly stated in  \cite{bertsimas2018adaptive}, ``there are
also problems that would generally not satisfy complete
recourse, such as a production planning problem
where a manager determines a production plan today
to satisfy all uncertain demands for tomorrow instead
of incurring penalty", see Example \ref{exmp-c} which satisfies relatively complete recourse.
	\item \textbf{The worst-case distribution:} In sharp contrast to those state-of-art models, this work derives the distribution attaining the worst-case second-stage cost with distribution uncertainty either in the objective function or constraints, respectively.
\end{itemize}

\section{Uncertainty in the Objective Function}
\label{uino}
We first consider the distribution uncertainty only in the objective function of \eqref{LO-Two} via the following form
\begin{equation}
	\begin{aligned}
		\label{Q_un_o}
		Q(\bm{x},\bm{\xi}) = \min \ \  &\bm{z}(\bm{\xi})^T\bm{y} \\
		\sta \ \ &A\bm{x}+B\bm{y} \ge \bm{b}\\
		& \bm{y} \in \mathbb{R}^{m}_+,
	\end{aligned}
\end{equation}
where $\bm{z}(\bm{\xi})$ is defined as \eqref{uncertainc} in Section \ref{cla-pro}.

We convert the problem in \eqref{primal} with $Q(\bm{x},\bm{\xi})$ given by \eqref{Q_un_o} over the 1-Wasserstein ball $\mathcal{F}_N$ to an SOCP problem which can be solved efficiently by general-purpose commercial-grade solvers such as CPLEX.

\begin{theo}
	\label{theo1}
	Under Assumptions \ref{rel-com}-\ref{M-set-assum}, the worst-case $\beta(\bm{x})$ with $Q(\bm{x},\bm{\xi})$ in \eqref{Q_un_o} over the 1-Wasserstein ball $\mathcal{F}_N$ is equivalent to the optimal value of an SOCP problem
	\begin{equation}
		\begin{aligned}
			\label{beta_dual_f}
			\beta(\bm{x}) = \inf \ \ &\left\{\lambda\epsilon_N+\frac{1}{N}\sum\limits_{i=1}^{N}s_i\right \}\\
			{\rm s.t.}\ \ & \lambda \ge \Vert {Z\bm{y}} \Vert_2 \\
			& s_i \ge \bm{z}_0^T\bm{y}+\bm{y}^TZ^T\widehat{\bm{\xi}}^i, \ \forall i\in [N]\\
			& A\bm{x}+B\bm{y} \ge \bm{b}, \ \  \bm{y}\ge \bm{0},
		\end{aligned}
	\end{equation}
	where $Z^T = [\bm{z}_1,\dots,\bm{z}_m]$.	
	
	Moreover, the associated DR problem (\ref{primal}) is equivalent to the following SOCP problem
	\begin{equation}
		\begin{aligned}
			\label{whole-problem_f}
			\minimize_{\bm{x} \in \mathcal{X}} ~~~ &\left\{\bm{c}^T\bm{x} + \lambda\epsilon_N+\frac{1}{N}\sum\limits_{i=1}^{N}s_i\right\} \\
			\st~~& \lambda \ge \Vert {Z\bm{y}} \Vert_2 \\
			& s_i \ge \bm{z}_0^T\bm{y} + \bm{y}^TZ^T\widehat{\bm{\xi}}^i, \  \forall i\in [N]\\
			& A\bm{x}+B\bm{y} \ge \bm{b}, \   \bm{y}\ge \bm{0}.\\
		\end{aligned}
	\end{equation}
	
\end{theo}

\begin{proof}
	For any feasible first-stage decision vector $\bm{x}$, $\beta(\bm{x})$ over the 1-Wasserstein ball can be obtained by solving a conic linear program
	\begin{equation}
		\begin{aligned}
		\label{beta-linear}
			\beta(\bm{x}) =\sup  \ \ & \sum\limits_{i=1}^{N}\int_{\Xi}Q(\bm{x},\bm{\xi})K(\mathrm{d}\bm{\xi},\widehat{\bm{\xi}}^i) \\
			\sta \ \ &\int_{\Xi}K(\mathrm{d}\bm{\xi},\widehat{\bm{\xi}}^i)=\frac{1}{N}, \forall i\in [N]\\
			&\int_{\Xi}\sum\limits_{i=1}^{N}d(\bm{\xi},\widehat{\bm{\xi}}^i)K(\mathrm{d}\bm{\xi},\widehat{\bm{\xi}}^i)\le \epsilon_N.
		\end{aligned}
	\end{equation}
	The Lagrange dual function for \eqref{beta-linear} is represented as
	\begin{align*}
	&g(\lambda,\bm{s}) \\
	 &= \sup_{{{\bm{\xi}}}\in \Xi} \left\{\int_{\Xi} \sum\limits_{i=1}^{N}\left(Q(\bm{x},\bm{\xi})-s_i -\lambda  d({\bm{\xi}},{\widehat{\bm{\xi}}}^i)\right) K(\mathrm{d}{\bm{\xi}},{\widehat{\bm{\xi}}}^i)\right\}
	 \\
&~~+\frac{1}{N}\sum\limits_{i=1}^{N}{s_i}+ \lambda \epsilon_N. 	
	\end{align*}
	Consequently, the dual problem of \eqref{beta-linear} is given as		
	\begin{align}
			\beta(\bm{x}) =\inf  \ \ & \lambda\epsilon_N + \frac{1}{N}\sum\limits_{i=1}^{N}s_i \label{beta_dual1_1} \\
			\sta \ & \lambda \ge 0 \nonumber\\
&\hspace*{-0.1in} Q(\bm{x},\bm{\xi}) - \lambda d(\bm{\xi},\widehat{\bm{\xi}}^i) \le s_i, \forall i \in [N], ~\bm{\xi}\in \Xi \label{beta_dual1_b_1}
			.
	\end{align}

	Since $\epsilon_N > 0$, then $K = \mathbb{F}_N \times \mathbb{F}_N$ is a strictly feasible solution to \eqref{beta-linear}, the Slater condition for the strong duality of primal problem \eqref{beta-linear} and its dual problem \eqref{beta_dual1_1} is satisfied  \cite{Shapiro2001On}.

	The constraints in \eqref{beta_dual1_b_1} require a feasible second-stage solution $\widehat{\bm{y}}$ to guarantee the feasibility of the following inequality
	\begin{equation*}
		\begin{aligned}
			\bm{z}(\bm{\xi})^T\widehat{\bm{y}}- \lambda d(\bm{\xi},\widehat{\bm{\xi}}^i) \le s_i, \forall i \in [N],~\bm{\xi}\in \Xi.
		\end{aligned}
	\end{equation*}
Note that Assumption \ref{rel-com} ensures the existence of such a $\widehat{\bm{y}}$. Hence, \eqref{beta_dual1_b_1} can be expressed as
	\begin{equation}
		\begin{aligned}
			s_i \ge \bm{z}(\bm{\xi})^T\widehat{\bm{y}}- \lambda d(\bm{\xi},\widehat{\bm{\xi}}^i),\ \ \forall  i \in [N], ~\bm{\xi}\in \Xi.
		\end{aligned}
	\end{equation}
	Since
	 \begin{equation*}
	 \bm{z}(\bm{\xi})^T\widehat{\bm{y}} = \left(\bm{z}_0+\sum_{i=1}^{m}\xi_i\bm{z}_i\right)^T\widehat{\bm{y}}=\bm{z}_0^T\widehat{\bm{y}}+\bm{\xi}^TZ\widehat{\bm{y}},
	 \end{equation*}
	 it implies that
	\begin{equation*}
		\begin{aligned}		
			&\sup_{\bm{\xi}} \left\{\bm{z}(\bm{\xi})^T\widehat{\bm{y}}-\lambda\Vert \bm{\xi}-\widehat{\bm{\xi}}^i\Vert_2 \right\}\\
			&=\sup_{\bm{\xi}} \left\{  \bm{z}_0^T\widehat{\bm{y}}+\bm{\xi}^TZ\widehat{\bm{y}}-\lambda\Vert \bm{\xi}-\widehat{\bm{\xi}}^i\Vert_2\right\} \\	
			&=\left\{	
			\begin{array}{ll}
				\bm{z}_0^T\widehat{\bm{y}}+\widehat{\bm{y}}^TZ^T\bm{\xi}^i, &\text{if} \   {\| Z\widehat{\bm{y}} \|}_2 \le \lambda  \\
				+\infty, &\text{if} \   {\| Z\widehat{\bm{y}} \|}_2 > \lambda  \\
			\end{array}
			\right.
		\end{aligned}
	\end{equation*}
	where the last equality follows from Lemma 1 in \cite{wang2020wasserstein}.
	
	Consequently, \eqref{beta_dual1_b_1} admits an equivalent form
	\begin{equation*}
		\left\{
		\begin{aligned}
			& s_i \ge \bm{z}_0^T\widehat{\bm{y}}+\widehat{\bm{y}}^TZ^T\widehat{\bm{\xi}}^i, ~\forall i \in [N], \\
			&\lambda \ge {\| Z\widehat{\bm{y}} \|}_2,
		\end{aligned}
		\right.
	\end{equation*}	
	Inserting the above to \eqref{beta_dual1_b_1} leads to the equivalence of \eqref{beta_dual_f} and  \eqref{ADRO-Two}. Hence, the two-stage problem \eqref{primal} can be equivalently reformulated as the SOCP problem \eqref{whole-problem_f}.
 \end{proof}
Theorem \ref{theo1} shows that the optimization program \eqref{primal} can be reformulated as a tractable SOCP problem. Furthermore, different $l_p$-norms in \eqref{defdis} lead to different equivalent forms of the DR two-stage problem, see Table \ref{form} for details, where LP represents the linear programming.


\renewcommand{\arraystretch}{1.2}
\begin{table}[htb]
	\centering
	\caption{Equivalent problems of the our DR problem, where $p$ represents the $l_p$-norm in \eqref{defdis}.}
    \begin{tabu} to 0.48\textwidth{|X[1.5,c]| X[1,c]| X[1,c]| X[1.5,c]| X[3,c]|}
        \hline
        Norm&$p=1$&$p=2$&$p=\infty$&Otherwise\\
        \hline
        Problem&LP&SOCP&LP&Convex Program\\
        \hline
    \end{tabu}
	\label{form}
\end{table}

\section{Uncertainty in the Constraints}
\label{uinc}
In this section we consider the distribution uncertainty only in constraints of \eqref{LO-Two}, i.e.,
\begin{equation}
	\begin{aligned}
		\label{Q_un_cons}
		Q(\bm{x},\bm{\xi}) = \min \ \  &\bm{z}^T\bm{y} \\
		\sta\ \ &A(\bm{\xi})\bm{x}+B\bm{y} \ge \bm{b}(\bm{\xi})\\
		& \bm{y} \in \mathbb{R}^{m}_+,
	\end{aligned}
\end{equation}
where $A(\bm{\xi})$ and $\bm{b}(\bm{\xi})$ are defined in \eqref{uncertainc} of Section \ref{cla-pro}.

\subsection{Reformulation of the DR Problem}\label{CGP}
We first prove the NP-hardness of the problem \eqref{primal} with $Q(\bm{x},\bm{\xi})$ given in \eqref{Q_un_cons}.
\begin{theo}
	\label{theo2}
	Under Assumptions \ref{rel-com}-\ref{M-set-assum}, the worst-case  $\beta(\bm{x})$ with $Q(\bm{x},\bm{\xi})$ in \eqref{Q_un_cons} over the 1-Wasserstein ball $\mathcal{F}_N$  can be computed by an NP-hard problem
	\begin{align}
			\beta(\bm{x}) = \inf \ \ &\left\{\lambda\epsilon_N+\frac{1}{N}\sum\limits_{i=1}^{N}s_i \right\}\label{beta_dual_c} \\
			{\rm s.t.} \ \ & s_i \ge ({C\bm{p}})^T \widehat{\bm{\xi}}^i + \bm{p}^T(\bm{b}_0-A_0\bm{x}) \label{wh_LP_con}\\
&\lambda \ge \| {C\bm{p}} \|_2,
			~\forall i\in [N], ~\bm{p}\in\mathcal{P}, \label{wp1}
	\end{align}
	where
    \begin{equation}
	   \label{matrix-c}
		  C = [\bm{b}_1-A_1{\bm{x}},\dots,\bm{b}_{m}-A_{m}{\bm{x}}]^T
	\end{equation}
 and $\mathcal{P}$ is a polyhedron given by
	\begin{equation}
	\label{poly}
	\mathcal{P} = \{\bm{p}\in \mathbb{R}^k_+:B^T\bm{p}\le\bm{d}\}.
	\end{equation}

\end{theo}

\begin{proof}
	The strong duality stills holds for $\beta(\bm{x})$, which is rewritten  as		
	\begin{align}
			\beta(\bm{x}) =\inf  \ \ & \lambda\epsilon_N + \frac{1}{N}\sum\limits_{i=1}^{N}s_i \label{beta_dual1_c} \\
			\sta \ & \lambda \ge 0 \nonumber \\
        &\hspace*{-0.1in} Q(\bm{x},\bm{\xi}) - \lambda d(\bm{\xi},\widehat{\bm{\xi}}^i) \le s_i, \forall i \in [N], ~\bm{\xi}\in \Xi. \label{beta_dual1_c_1}
	\end{align}
	Under the strong duality of the LP problem, $Q(\bm{x},\bm{\xi})$ in \eqref{Q_un_cons} is equivalent to
	\begin{equation}
		\begin{aligned}
			\label{Qx-dual1}
			Q(\bm{x},\bm{\xi}) =\max \ \ & \bm{p}^T(\bm{b}(\bm{\xi})-A(\bm{\xi})\bm{x})	\\
			\sta \ \  & \bm{z}\ge B^T\bm{p} \\
			& \bm{p} \ge 0.
		\end{aligned}
	\end{equation}
	Then the constraints in \eqref{beta_dual1_c_1} can be expressed as
	\begin{equation}
		\begin{aligned}
			\label{beta_dual1_b_2}
			\hspace{-.3cm}s_i \ge \bm{p}^T\left(\bm{b}(\bm{\xi})-A(\bm{\xi})\bm{x}\right)-\lambda d(\bm{\xi},\widehat{\bm{\xi}}^i),\forall \bm{\xi}\in \Xi, \bm{p}\in \mathcal{P}.
		\end{aligned}
	\end{equation}
	Furthermore, the right-hand side of \eqref{beta_dual1_b_2} is expressed as
	\begin{equation*}
		\begin{aligned}
		&\sup_{ \bm{\xi}}\left\{\bm{p}^T\left(\bm{b}(\bm{\xi})-A(\bm{\xi})\bm{x}\right)-\lambda d(\bm{\xi},\widehat{\bm{\xi}}^i)\right\} \\
		&=\sup_{ \bm{\xi}}\left\{(C\bm{p})^T\bm{\xi} + \bm{p}^T(\bm{b}_0-A_0\bm{x})-\lambda d(\bm{\xi},\widehat{\bm{\xi}}^i)\right\}\\
		&=\left\{	
		\begin{array}{ll}
		({C\bm{p}})^T\widehat{\bm{\xi}}^i -\bm{p}^T(\bm{b}_0-A_0\bm{x}),  &\text{if} \   {\| {C\bm{p}} \|}_2 \le \lambda  \\
		+\infty,  &\text{if} \   {\|{C\bm{p}} \|}_2 > \lambda,  \\
		\end{array}
		\right.
		\end{aligned}
	\end{equation*}	
	where $C$ is defined in \eqref{matrix-c} and the second equality follows from Lemma 1 in \cite{wang2020wasserstein}.
	
	Consequently, \eqref{beta_dual1_c_1} is equivalent to
	\begin{equation*}
		\left\{
		\begin{array}{ll}
			s_i \ge ({C\bm{p}})^T\widehat{\bm{\xi}}^i -\bm{p}^T(\bm{b}^0-A^0\bm{x}), &\forall i \in [N], ~\bm{p}\in \mathcal{P} \\
			\lambda  \ge {\| {C\bm{p}} \|}_2,  &\forall \bm{p}\in \mathcal{P}.
		\end{array}
		\right.
	\end{equation*}	
Thus, $\beta(\bm{x})$ in \eqref{ADRO-Two} is reformulated as \eqref{beta_dual_c}.
	
The constraint \eqref{wp1} in \eqref{beta_dual_c} can be expressed as
	$$ \lambda \ge \max_{\bm{p} \in \mathcal{P}}\| {C\bm{p}} \|_2. $$
Thus, the norm maximization problem over the polyhedron is NP-complete  \cite{bodlaender1990computational} and checking the feasibility of constraint \eqref{wp1} is NP-hard. This completes the proof.  \end{proof}

Theorem \ref{theo2} immediately implies the NP-hardness of the problem in \eqref{primal}. If the extreme point set $\mathcal{E}$ of the polyhedron $\mathcal{P}$ is explicitly known, the problem \eqref{primal} can be reformulated as a solvable SOCP problem.

\begin{coro}
	\label{sepcial-case}
	Suppose that Assumptions \ref{rel-com}-\ref{M-set-assum} hold and the extreme point set $\mathcal{E}$ of the polyhedron $\mathcal{P}$ in \eqref{poly} is known, the 1-Wasserstein problem \eqref{primal} with $Q(\bm{x},\bm{\xi})$ in \eqref{Q_un_cons} is equivalent to an SOCP problem
	\begin{equation}
		\label{whole-problem_c_s}
		\begin{aligned}
			\minimize_{\bm{x} \in \mathcal{X}} ~~~ &\left\{\bm{c}^T\bm{x}+\lambda\epsilon_N+\frac{1}{N}\sum\limits_{i=1}^{N}s_i \right\}\\
			\st ~~ &  s_i \ge (\bm{C}\bm{p})^T \widehat{\bm{\xi}}^i + \bm{p}^T(\bm{b}^0-A^0\bm{x}), \\
			& \lambda \ge \| \bm{C}\bm{p} \|_2, \ \ \forall i\in [N], ~\bm{p} \in \mathcal{E}.\\
		\end{aligned}
	\end{equation}
\end{coro}
\begin{proof}
 Since the LP problem \eqref{Qx-dual1} attains its optimal value at an extreme point of its feasible set $\mathcal{P}$, it holds that
	\begin{equation*}
		\begin{aligned}
			Q(\bm{x},\bm{\xi}) =\max_{\bm{p} \in \mathcal{E}}  \ \bm{p}^T(\bm{b}(\bm{\xi})-A(\bm{\xi})\bm{x}).	
		\end{aligned}
	\end{equation*}

Then the constraints in \eqref{wp1} and \eqref{wh_LP_con} can be explicitly expressed as
	\begin{equation*}
		\left\{
		\begin{array}{ll}
			s_i \ge (C\bm{p})^T\widehat{\bm{\xi}}^i -\bm{p}^T(\bm{b}^0-A^0\bm{x}), &\forall  i \in [N], ~ \bm{p} \in \mathcal{E}, \\
			\lambda  \ge {\|C\bm{p} \|}_2,  &\forall  \bm{p} \in \mathcal{E},
		\end{array}
		\right.
	\end{equation*}
which leads to the equivalence of \eqref{whole-problem_c_s} and \eqref{primal}. This completes the proof.
 \end{proof}


Corollary \ref{sepcial-case} shows that we can solve the DR two-stage problem by explicitly enumerating the extreme points of the polyhedron $\mathcal{P}$. Motivated by this, we design an algorithm to approximately solve the NP-hard DR two-stage problem via a constraint generation approach.

\subsection{Approximately Solving the DR Two-stage Problem with Uncertainty in Constraints} \label{sol2uc}

In this subsection, we propose a constraint generation algorithm to solve \eqref{primal}. Inspired by Corollary \ref{sepcial-case}, the DR problem can be efficiently solved given all extreme points of $\mathcal{P}$. While the direct enumeration of all extreme points is computational demanding, we gradually select sets of ``good" extreme points by solving a sequence of second-stage problems $\beta(\bm{x})$. Particularly, we utilize a master-subproblem framework to approximately solve \eqref{primal}.

In the master problem (MP), we find an optimal solution under a selected subset of extreme points. Then a subproblem (SuP) is solved to obtain a better subset of extreme points. We add these points to the subset in MP as feasible cuts. Note that the optimal values of the MP and SuP are the lower and upper bounds for \eqref{primal} respectively. Both the lower and upper bounds will converge and a good solution to \eqref{primal} can be obtained. The algorithm based on such an MP-SuP framework is given in the sequel.

By Corollary \ref{sepcial-case}, the MP is an SOCP problem given as
\begin{align}
		\minimize_{\bm{x} \in \mathcal{X}} \ \ &\left\{\bm{c}^T\bm{x}+\lambda\epsilon_N+\frac{1}{N}\sum\limits_{i=1}^{N}s_i \right\} \label{MP} \\
		\st \ & s_i \ge ({C\bm{p}})^T \widehat{\bm{\xi}^i} + \bm{p}^T(\bm{b}^0-A^0\bm{x}), \nonumber\\
		& \lambda \ge \| {C\bm{p}} \|_2, \ \ \forall i\in [N], ~\bm{p} \in \mathcal{E}_s,\nonumber
\end{align}
where $\mathcal{E}_s$ is a given subset of extreme points of $\mathcal{P}$.

After obtaining an optimal solution $\bm{x}^m$ of the MP, an SuP is derived as follows
\begin{align}
		\beta(\bm{x}^m)=\min_{\lambda^s, s^s_i} \ \ &\left\{\lambda^s\epsilon_N+\frac{1}{N}\sum\limits_{i=1}^{N}s^s_i\right\} \label{SP} \\
		\sta\ \ \ & s^s_i \ge ({C\bm{p}})^T \widehat{\bm{\xi}^i} + \bm{p}^T(\bm{b}^0-A^0{\bm{x}^m}), \nonumber\\
		& \lambda^s \ge \Vert {C\bm{p}} \Vert_2, \ \ \forall i\in [N]~\bm{p}\in \mathcal{P}. \nonumber
\end{align}

A weak condition is needed to obtain an good solution of the SuP.
\begin{assum}
	\label{poly-bound}
	The polyhedron $\mathcal{P}=\{\bm{p}\in \mathbb{R}^k_+:B^T\bm{p}\le\bm{z}\}$ is nonempty and bounded.
\end{assum}
The decision variables $\lambda^s$ and $\bm{s}^s$ in \eqref{SP} are completely decoupled and hence we can find their optimal solutions separately. To achieve it, we have the following steps.

\begin{enumerate}
\item
An optimal solution $\bm{s}^s$ to SuP is obtained by solving a group of linear programs, i.e.,
\begin{equation}
	\begin{aligned}
		\label{sub_lp}
		s^s_i =\max \ \ & (C\bm{p})^T\widehat{\bm{\xi}^i}+\bm{p}^T(\bm{b}^0-A^0\bm{x}^m) \\
		\sta \ \ & \bm{p} \in \mathcal{P}. \\
	\end{aligned}
\end{equation}

\item An optimal $\lambda^s$ is obtained by solving a norm maximization problem, i.e.,
\begin{equation}
	\begin{aligned}
		\label{sub_norm}
		\lambda^s = &\max \ \   \| C\bm{p} \|_2 \\
		&\sta \ \  \bm{p} \in \mathcal{P}.
	\end{aligned}
\end{equation}
\end{enumerate}

A sequence of optimal solutions $\{\bm{p}_i^*\}_{i=1}^N$ to \eqref{sub_lp} can be added to the extreme point subset $\mathcal{E}_s$ in MP, since the LP problem \eqref{sub_lp} obtains its optimal value at extreme points of the feasible region $\mathcal{P}$.

To solve the non-convex norm maximization problem, we adopt the consensus alternating direction method of multipliers (ADMM) method \cite{huang2016consensus}. Particularly, \eqref{sub_norm} is reformulated  as a consensus form via $m$ auxiliary variables $\{\bm{g}_1,\dots,\bm{g}_m\}$, i.e.,
\begin{equation}
\begin{aligned}
\label{sub_norm_cons}
\lambda^s = &\min \ \   -\bm{p}^TC^TC\bm{p} \\
&\ \sta \ \  \bm{b}^T_i\bm{g}_i \le z_i, \bm{g}_i\ge \bm{0}\\
& \ \ \ \ \ \ \ \ \bm{g}_i = \bm{p},~ \forall i \in [m],
\end{aligned}
\end{equation}
where $\bm{b}_i$ is the $i$-th column of $B$. Algorithm \ref{algo_norm} provides the detailed consensus-ADMM algorithm. We omit its convergence proof for brevity, which can be found in  \cite{huang2016consensus}.


\begin{algorithm}[t!]
	\caption{The consensus-ADMM for \eqref{sub_norm_cons}}
	\label{algo_norm}
	\begin{algorithmic}[1]
        \REQUIRE {Matrix $B,C$, vector $\bm{z},\bm{g}_i$ and $\bm{u}_i$, tolerance $\tau$}
        \ENSURE {Optimal solution $\bm{p}^*$ and optimal value $\lambda^s$}
		\STATE {Initialize $\bm{g}_i$ and $\bm{u}_i$}
		\REPEAT
		\STATE $\bm{p} \leftarrow\left(\frac{-C^TC}{\rho}+mI\right)(\sum_{i=1}^{m}(\bm{g}_i+\bm{u}_i))$
		\FOR{each $i \in [m]$}
		\STATE {	
		$\bm{g}_i\leftarrow  \arg \min_{\bm{z}_i} \|\bm{g}_i-\bm{p}+\bm{u}_i\|^2$\\
		$\ \ \ \ \ \ \st \  \bm{b}^T_i\bm{g}_i \le z_i,~\bm{g}_i\ge \bm{0}$\\
		$\bm{u}_i \leftarrow \bm{g}_i+\bm{u}_i-\bm{p}$
		}	\
		\ENDFOR
		\UNTIL{The successive difference of $\bm{p}$ is smaller than $\tau$}
    \STATE{Return $\bm{p}^* \leftarrow \bm{p}$ and $\lambda^s \leftarrow \|C\bm{p}^*\|_2$}		
	\end{algorithmic}
\end{algorithm}

By Assumption \ref{poly-bound}, a solution $\bm{p}^*$ to \eqref{sub_norm} as an extreme point of polyhedron $\mathcal{P}$ is ensured to exist and then is added to the  subset $\mathcal{E}_s$ \cite{bodlaender1990computational}.

\begin{algorithm}[t!]
	\caption{Solve the robust program }
	\label{algo_whole}
	\begin{algorithmic}[1]
    \REQUIRE{ A set of extreme points, $UB = +\infty$, $LB = -\infty$, $~~~~~k=0$}\\
	\ENSURE{ Optimal solution $\bm{x}^*$}
	\REPEAT
    \STATE{Add extreme points to $\mathcal{E}_s$ in \eqref{MP} and set $k = k+1$}
	\STATE{Solve \eqref{MP} to obtain an optimal solution $\{{\bm{x}}_k,{\bm{s}}_k,{\lambda}_k\}$ and set $$LB = \bm{c}^T\bm{x}_{k} + {\lambda}_k\epsilon_N + \frac{1}{N}\sum_{i=1}^{N}s_{ki}$$}
	\STATE {Solve \eqref{SP} to obtain an optimal solution $\{\bm{s}^s_{k},\lambda^s_k\}$ and extreme points $\{\bm{p}_k^{i}\}_{i=1}^N \cup \{\bm{p}_k\}$ and set $$UB = \min\{UB, \bm{c}^T\bm{x}_{k} + \lambda^s
		_k\epsilon_N+ \frac{1}{N}\sum_{i=1}^{N}s_{ki}^s\}$$}\
	\UNTIL{$UB - LB \le  \varepsilon$}	
	\STATE{Return $\bm{x}^* \leftarrow \bm{x}_k$}		
	\end{algorithmic}		
\end{algorithm}

We provide the MP-SuP based algorithm in  Algorithm \ref{algo_whole}. Theorem \ref{theo3} shows that Algorithm \ref{algo_whole} terminates in a finite number of iterations.

\begin{theo}
	\label{theo3}Under Assumption \ref{poly-bound},  Algorithm \ref{algo_whole} generates an optimal solution  of \eqref{primal} in $O(|\mathcal{E}|)$ iterations.
\end{theo}

\begin{proof}
	Let $\{\bm{x}_k,\lambda_k,\bm{s}_k\}$ be an optimal solution of MP in the $k$-th iteration and $\{\lambda_k^{s},\bm{s}_k^{s}\}$ be an optimal solution of SuP with $\{\bm{p}_k^{i}\}_{i=1}^N \cup \{\bm{p}_k\}$ being the extreme points of SuP. We show that $\{\bm{p}_k^{i}\}_{i=1}^N \cup \{\bm{p}_k\} \subseteq \mathcal{E}_s$ implies the convergence of Algorithm \ref{algo_whole}, i.e., $LB = UB$.
	
Step 4 in Algorithm \ref{algo_whole} implies that
	\begin{equation*}
        UB \le \bm{c}^T\bm{x}_k + \frac{1}{N}\sum_{i=1}^{N}s_{ki}^{s}+\epsilon_N\lambda_k^{s}.
    \end{equation*}
Since $\{\bm{p}_k^{i}\}_{i=1}^N \cup \{\bm{p}_k\} \subseteq \mathcal{E}_s$, then MP in the $k$-th iteration is identical to that in the $(k-1)$-th iteration. Thus, $\bm{x}_k$ is an optimal solution to the $(k-1)$-th MP as well. By the Step 3 in Algorithm \ref{algo_whole}, we find that
$LB \ge \bm{c}^T\bm{x}_k + \epsilon_N\lambda_k+\sum_{i=1}^{N}\frac{s_{ki}}{N} \ge \bm{c}^T\bm{x}_k +\epsilon_N\lambda_k^{s}+ \sum_{i=1}^{N}\frac{s_{ki}^{s}}{N},$
where the last inequality holds due to the fact that $\{\bm{p}_k^{i}\}_{i=1}^N \cup \{\bm{p}_k\} \subseteq \mathcal{E}_s$ and hence the related constraints are added to MP before  the $(k-1)$-th iteration. Consequently, we have $UB = LB$.
	
The conclusion of the convergence in $O(|\mathcal{E}|)$ iterations follows immediately from the finite number of extreme points for the polyhedron $\mathcal{P}$.
 \end{proof}

\section{The Worst-case Distribution and the Asymptotic Consistency}\label{wos-cas}

\subsection{The Worst-case Distribution}
In this subsection we derive the distribution achieving the worst-case $\beta(\bm{x})$ in \eqref{ADRO-Two} of Section \ref{TDROLOP} for any feasible vector $\bm{x} \in \mathcal{X}$.

\begin{lem}
	\label{thm_dis}
	For any feasible first-stage decision vector $\bm{x}$, then
	\begin{equation}
		\label{worst-case-B}
		\beta(\bm{x})=\sup_{\tilde{\bm{\xi}}\in \mathcal{B}}\left\{\frac{1}{N}\sum_{i=1}^{N}Q(\bm{x},\bm{\xi}^{(i)})\right\},
	\end{equation}
	where
$$
		\label{setB}
		\mathcal{B} = \left\{(\bm{\xi}^{(1)},\dots,\bm{\xi}^{(N)}) ~|~   \frac{1}{N}\sum_{i=1}^{N}d(\bm{\xi^{(i)}},\widehat{\bm{\xi}}^{i})\le\epsilon_N,\  \bm{\xi}^{(i)}\in \Xi\right\}.
$$
\end{lem}
\begin{proof}
	Given a feasible solution $\bm{x}$, it follows that
	\begin{equation}
    \label{s1}
		\begin{aligned}
			\sup_{\tilde{\bm{\xi}}\in \mathcal{B}}\left\{\frac{1}{N}\sum_{i=1}^{N}Q(\bm{x},\bm{\xi}^{(i)})\right\} \le \sup_{F\in \mathcal{F}_N}\mathbb{E}_\mathbb{F} \left\{Q(\bm{x},\bm{\xi})\right\},
		\end{aligned}	
	\end{equation}
	by Lemma 2 in \cite{wang2020wasserstein}.
	
By the equivalence between $\beta(\bm{x})$ and \eqref{beta_dual1_1}, then for any $\varepsilon \ge 0$, there exists $\{\tilde{\bm{\xi}}^{(i)}\}_{i\in[N]}\subseteq \Xi$ such that
	\begin{equation}\label{contradict}
		\begin{aligned}
			& \sup_{F\in \mathcal{F}_N}\mathbb{E}_\mathbb{F} \left\{Q(\bm{x},\bm{\xi})\right\}-\varepsilon  \\
 & <  \inf_{\lambda \ge 0}\left\{\lambda\epsilon_N +	 \frac{1}{N}\sum_{i=1}^{N}\left\{Q(\bm{x},\tilde{\bm{\xi}}^{(i)})-\lambda d(\tilde{\bm{\xi}}^{(i)},\widehat{\bm{\xi}}^{i})\right\} \right\}.
		\end{aligned}
	\end{equation}

	If $\left(\tilde{\bm{\xi}}^{(1)},\dots,\tilde{\bm{\xi}}^{(N)}\right) \notin \mathcal{B}$ and let $\lambda > 0 $, it follows that
	\begin{equation*}
		\lambda\left\{\epsilon_N-\frac{1}{N} \sum_{i=1}^{N}d(\tilde{\bm{\xi}}^{(i)},\widehat{\bm{\xi}}^i)\right\} < 0.
	\end{equation*}
	Increasing $\lambda$ to $+\infty$ in \eqref{contradict} enforces $\sup_{\mathbb{F}\in \mathcal{F}_N}\mathbb{E}_\mathbb{F} \{Q(\bm{x},\bm{\xi})\}$ to $-\infty$, which contradicts with the fact that
	$$\sup_{\mathbb{F}\in \mathcal{F}_N}\mathbb{E}_\mathbb{F} \{Q(\bm{x},\bm{\xi})\} \ge \mathbb{E}_{\mathbb{F}_N} \{Q(\bm{x},\bm{\xi})\} > -\infty, $$
	where the second inequality follows from  Assumption \ref{rel-com}. Thus, $\left(\tilde{\bm{\xi}}^{(1)},\dots, \tilde{\bm{\xi}}^{(N)}\right) \in \mathcal{B}$.
	
	By Lemma 2 in \cite{wang2020wasserstein}, it holds that
	\begin{equation*}
		\begin{aligned}
			\sup_{\mathbb{F}\in \mathcal{F}_N}\mathbb{E}_\mathbb{F} \{Q(\bm{x},\bm{\xi})\}-\varepsilon < \sup_{\tilde{\bm{\xi}}\in \mathcal{B}}\left\{\frac{1}{N}\sum_{i=1}^{N}\left\{Q(\bm{x},{\bm{\xi}}^{(i)})\right\} \right\}.
		\end{aligned}
	\end{equation*}
    Letting $\varepsilon$ to zero, it holds that
	\begin{equation*}
		\begin{aligned}
			\sup_{F\in \mathcal{F}_N}\mathbb{E}_F \{Q(\bm{x},\bm{\xi})\le\sup_{\tilde{\bm{\xi}}\in \mathcal{B}}\left\{\frac{1}{N}\sum_{i=1}^{N}Q(\bm{x},\bm{\xi}^{(i)})\right\}.
		\end{aligned}
	\end{equation*}
Jointly with \eqref{s1}, then \eqref{worst-case-B} holds.
\end{proof}

Since $Q(\bm{x},\bm{\xi})$ is concave with respect to $\bm{\xi}$ and $\mathcal{B}$ is a compact set, \eqref{worst-case-B} allows for an optimal solution, Then a worst-case distribution  is explicitly derived below.

\begin{theo}
	\label{worst-case-dis-1}
	For any solution $\bm{x} \in \mathcal{X}$ and let $\bm{\xi}_{\bm{x}}=\left(\bm{\xi}^{(1)}_{\bm{x}},\dots,\bm{\xi}^{(N)}_{\bm{x}}\right)$ be an optimal solution to \eqref{worst-case-B}. The following distribution
	\begin{equation*}
		\begin{aligned}
			\mathbb{F}^*_{\bm{x}} = \frac{1}{N}\sum_{i=1}^{N}\delta_{\bm{\xi}^{(i)}_{\bm{x}}}
		\end{aligned}
	\end{equation*}
	is the distribution achieving the worst-case second-stage cost , i.e.,
	\begin{equation*}
		\begin{aligned}
			\sup_{\mathbb{F}\in \mathcal{F}_N}\mathbb{E}_\mathbb{F}\left\{Q(\bm{x},{\bm{\xi}})\right\} = \mathbb{E}_{\mathbb{F}^*_{\bm{x}}}\left\{Q(\bm{x},{\bm{\xi}})\right\}.
		\end{aligned}
	\end{equation*}
\end{theo}

\begin{proof}
     Obviously, the following distribution
	\begin{equation*}
		\begin{aligned}
			\Pi_{\bm{x}} = \frac{1}{N}\sum_{i=1}^{N}\delta_{(\bm{\xi}^{(i)}_{\bm{x}},\widehat{\bm{\xi}}^i)}.
		\end{aligned}
	\end{equation*}
	is a joint distribution of $F_N$ and $\mathbb{F}^*_{\bm{x}}$. Then it holds that
	\begin{equation*}
		\begin{aligned}
			W(\mathbb{F}_N,\mathbb{F}^*_{\bm{x}}) &\le \int\left\|{\bm{\xi}}-{\bm{\xi}}^{\prime}\right\|_p \Pi_{\bm{x}}\left(\mathrm{d} {\bm{\xi}}, \mathrm{d} {\bm{\xi}}^{\prime}\right)\\
&=\frac{1}{N} \sum_{i=1}^{N} \| {\bm{\xi}}_{\bm{x}}^{(i)}-\widehat{\bm{\xi}}^{i}\|_p \le \epsilon_N,
		\end{aligned}
	\end{equation*}
	where the first inequality follows directly from the definition of the 1-Wasserstein metric and  the last inequality follows from the fact that
	$\left(\bm{\xi}^{(1)}_{\bm{x}},\dots,\bm{\xi}^{(N)}_{\bm{x}}\right) \in \mathcal{B}$. Hence, $\mathcal{F}_N$ includes the distribution $F^*_{\bm{x}}$. Thus, it yields that
	\begin{equation*}
		\begin{aligned}
			\sup_{\mathbb{F}\in \mathcal{F}_N}\mathbb{E}_\mathbb{F} \left\{Q(\bm{x},{\bm{\xi}})\right\}&\ge\mathbb{E}_{\mathbb{F}^*_{\bm{x}}}\left\{Q(\bm{x},{\bm{\xi}})\right\}
			=\frac{1}{N}\sum\limits_{i=1}^{N}Q(\bm{x},
			\bm{\xi}^{(i)}_{\bm{x}})\\
			&=\sup_{\mathbb{F}\in \mathcal{F}_N}\mathbb{E}_\mathbb{F} \left\{Q(\bm{x},\bm{\xi})\right\},
		\end{aligned}
	\end{equation*}
	where the last equality follows from Lemma \ref{thm_dis}. Hence,  $\mathbb{F}^*_{\bm{x}}$ is the desired worst-case distribution.
\end{proof}

\subsection{The Asymptotic Consistency}
This subsection studies the asymptotic consistency of the DR problem \eqref{primal} under a mild assumption.

\begin{assum}
	\label{light-dis}
	There exists a positive constant $c$ such that
	\begin{equation*}
	\int_{\Xi}\exp(\|\bm{\xi}\|^c_2)\mathbb{F}(\mathrm{d}\bm{\xi})< \infty.
	\end{equation*}
    for the true distribution $\mathbb{F}$.
\end{assum}
Under Assumptions \ref{rel-com}-\ref{light-dis}, we formalize the asymptotic consistency of the proposed DR problem below.
\begin{theo}
	\label{asy_con}
	Under Assumptions \ref{rel-com}-\ref{light-dis} and select $\beta_N \in (0,1)$ such that $\sum_{N=1}^{\infty}\beta_N \le \infty$. Let the 1-Wasserstein ball radius be
	\begin{equation*}
	\epsilon_N(\beta_N)=\left\{\begin{array}{ll}
	\left(\frac{\log(c_1\beta_N^{-1})}{c_2N}\right)^{1/\max\{n,2\}},  & \rm{if} \ \  N \ge \frac{\log(c_1\beta_N^{-1})}{c_2}  \\
	\left(\frac{\log(c_1\beta_N^{-1})}{c_2N}\right)^{1/c}, & \rm{if} \ \  N < \frac{\log(c_1\beta_N^{-1})}{c_2}
	\end{array}
	\right.
	\end{equation*}	
	 where $c_1$ and $c_2$ are positive constants related to the constant $c$ in Assumption \ref{light-dis}. Then the DR problem \eqref{primal} asymptotically converges to the stochastic problem \eqref {classical} almost surely when the sample number increases to infinity.
\end{theo}

\begin{proof}
		For the problem with distribution uncertainty only in the objective function, the relatively complete recourse implies that $Q(\bm{x},\bm{\xi})$ is feasible and finite. Then there exists a finite $\bm{y}$ such that $|~Q(\bm{x},\bm{\xi})~|=|~(Z\bm{y})^T\bm{\xi}~| \le \|Z\bm{y}\|_2\|\bm{\xi}\|_2 \le L(1+\| {\bm{\xi}} \|_2)$ for any $\bm{x} \in \mathcal{X}$ and ${\bm{\xi}} \in \Xi$, where $L \ge 0$ is a constant.
		
		For the case of the distribution uncertainty only in constraints, the strong duality of LP problem shows that $Q(\bm{x},\bm{\xi}) = (C\tilde{\bm{p}})^T\bm{\xi}$, where $C$ is given in \eqref{matrix-c} of Section \ref{CGP} and $\tilde{\bm{p}}$ is the extreme point of polyhedron $\mathcal{P}$.
		Assumption \ref{poly-bound} implies that  $\|\tilde{\bm{p}}\|$ is bounded and hence there exists a positive constant $L$ such that $|~Q(\bm{x},\bm{\xi})~| \le \|C\tilde{\bm{p}}\|_2\|\bm{\xi}\|_2 \le L(1+\| {\bm{\xi}} \|_2)$ for $\bm{x} \in \mathcal{X}$ and ${\bm{\xi}} \in \Xi$.

Finally the asymptotic consistency of our model follows from Theorem $3.6$ in \cite{esfahani2018data}.
 \end{proof}
%
%

\section{Simulation}\label{sim}

This section conducts  experiments to evaluate the performance of the proposed model and the constraint generation algorithm.  All experiments are performed on a 64 bit PC with an Intel Core i5-7500 CPU at 3.4GHz and 8 GB RAM. The Cplex 12.6 optimizer is used to solve the optimization programs.

\subsection{The Two-stage Portfolio Program}
This subsection is devoted to the application in two-stage portfolio program with uncertainty only in the objective function as stated in Example \ref{exmp-f}, see \cite{ling2017robust} for details.

\subsubsection{Problem Specification}
Consider a portfolio of four assets: (1) Dow Jones Industrial Average Index, (2) Dow Jones Transportation Average Index, (3) Dow Jones Composite Average Index and (4) Dow Jones Utility Average. The daily returns of above assets over seven years from January 02th, 2011 to December 31th, 2018 are collected from the RESSET database (http://www.resset.cn).


Since the first-stage return $\bm{c}$ is unknown in our simulation, we select the data from January 02th, 2011 to December 31th, 2016 to approximate  it by the SAA method, i.e.,  $\bm{c}=\sum_{i=1}^N \widehat{\bm{\xi}^1_i}$, where $\widehat{\bm{\xi}}^1_i$ is the $i$th sample of the first-stage return.

\subsubsection{Impact of the 1-Wasserstein Radius and the Sample Size}
Experiments are conducted to test the impact of the 1-Wasserstein radius $\epsilon_N$ and the sample size $N$ on the out-of-sample performance of our model in this subsection. The out-of-sample performance is measured by the loss of the proposed model on {\em new} samples, i.e.,
 \begin{equation}
 \label{Out_P}
   \bm{c}^T\bm{x} + \mathbb{E}_{\mathbb{F}}\{Q(\bm{x},\bm{\xi})\}.
 \end{equation}
We are unable to exactly calculate \eqref{Out_P} due to the unknown true distribution $\mathbb{F}$. Instead, we randomly choose $300$ test samples from the dataset to approximate it, i.e.,
 \begin{equation*}
  \bm{c}^T\bm{x} + \frac{1}{ N_T}\sum\limits_{i=1}^{N_T}Q(\bm{x},\widehat{\bm{\xi}}_T^i),
 \end{equation*}
where $\widehat{\bm{\xi}}_T^i$ is the $i$-th test sample and $N_T$ is the number of test samples.

We first test the impact of the 1-Wasserstein radius $\epsilon_N$ on our model. We conduct $200$ independent experiments and the averaged out-of-sample performance is illustrated in Figure \ref{radius}. Experimental results show that the out-of-sample performance improves as the 1-Wasserstein radius increases and decreases if the radius is greater than a specific value.

\begin{figure*}[htbp]
	\centering
	\subfigure[ ]{
		\label{fig:30 }
		\includegraphics[width=0.3 \textwidth]{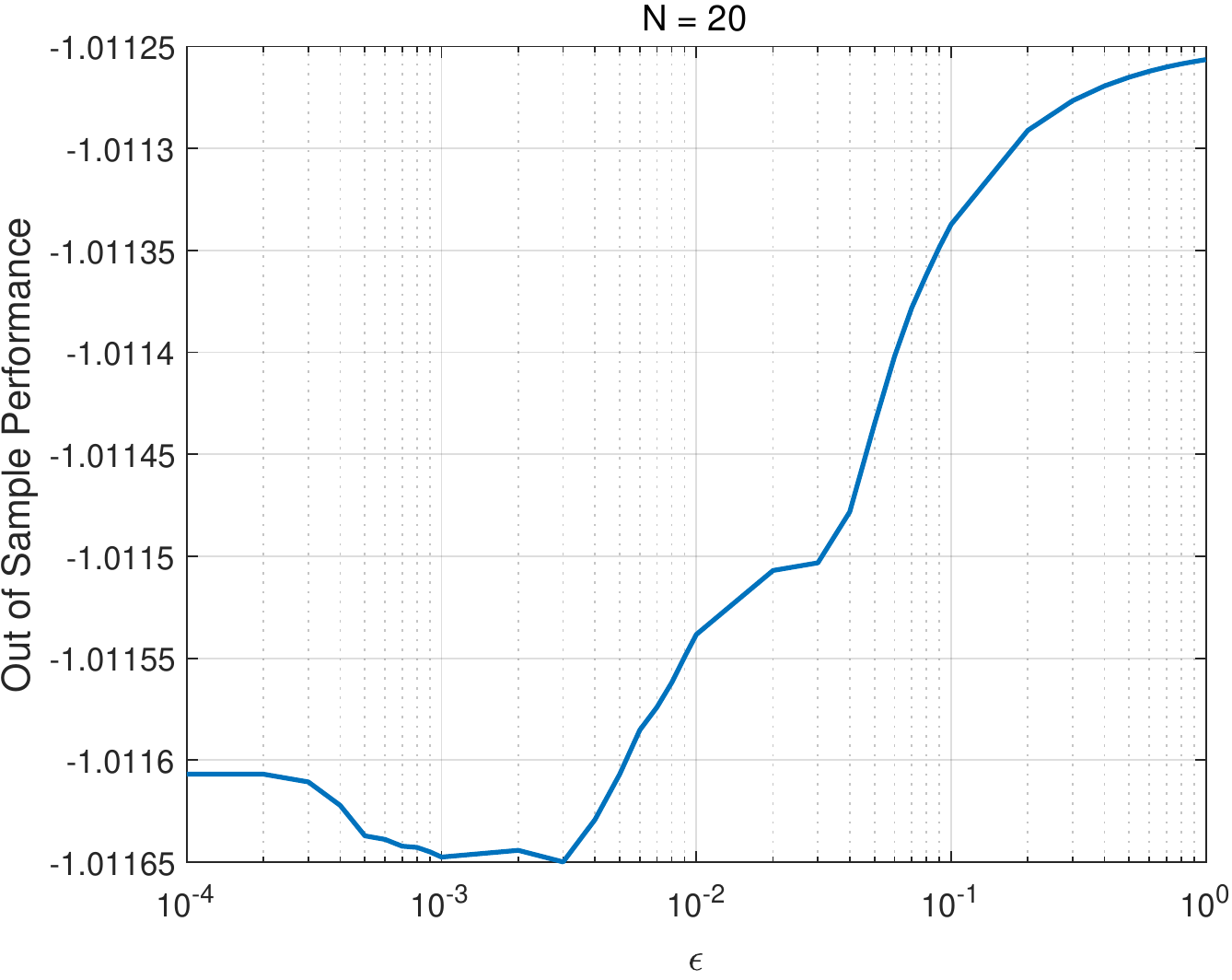}
	}
	\subfigure[ ]{
		\label{fig:100 }
		\includegraphics[width=0.3 \textwidth]{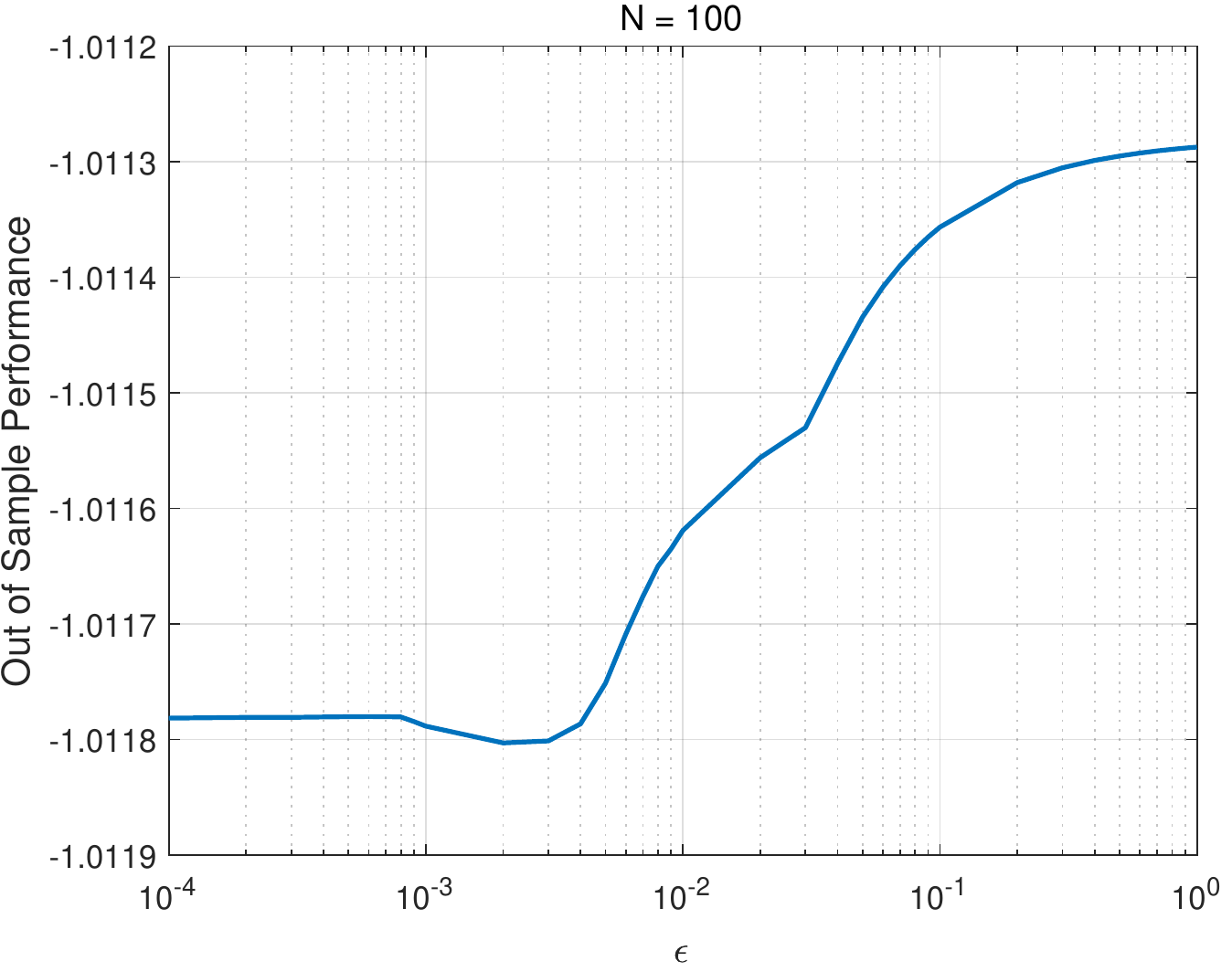}
	}
	\subfigure[ ]{
		\label{fig:300}
		\includegraphics[width=0.3 \textwidth]{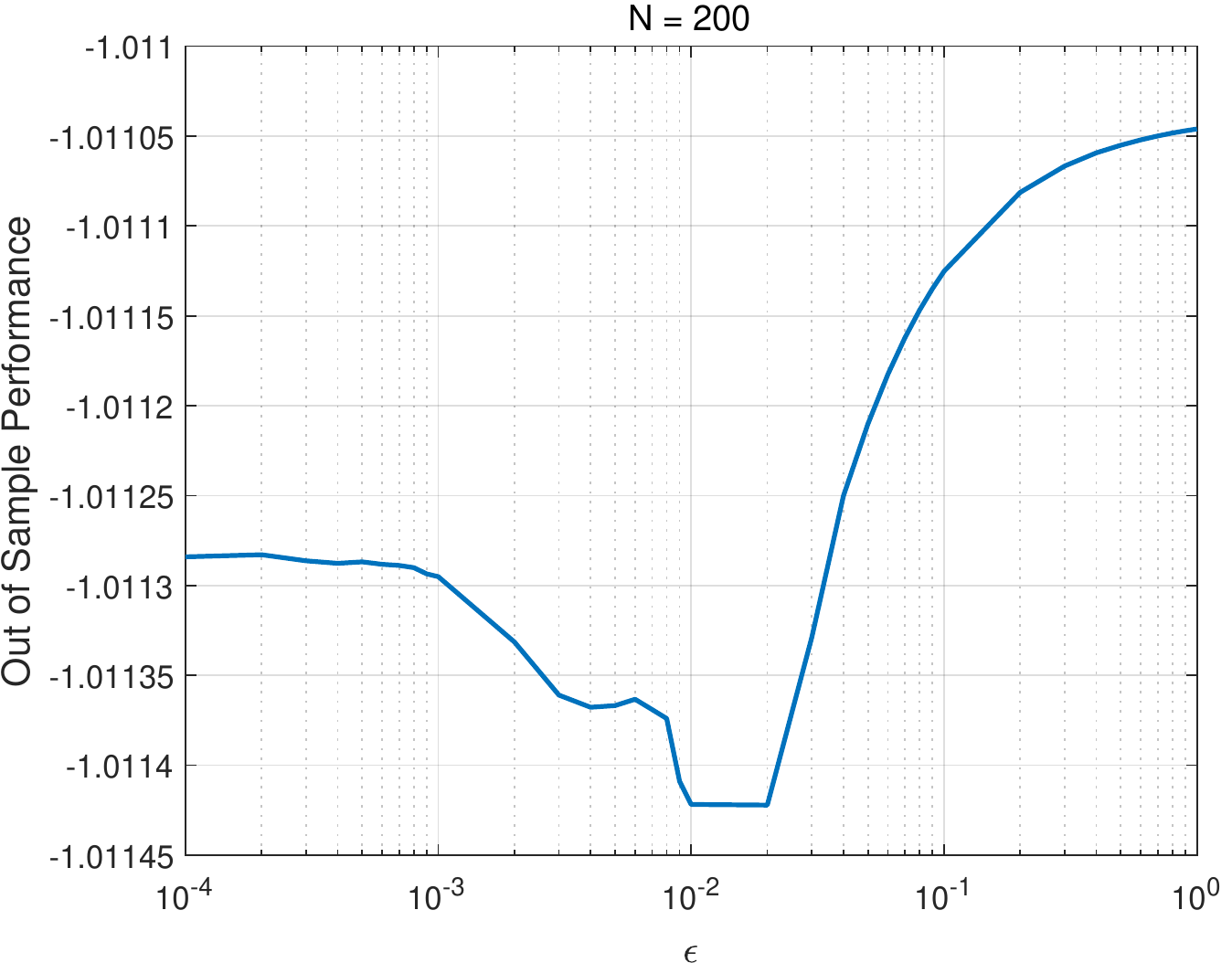}
	}
	\caption{{The averaged out-of-sample performance under sample dataset of different sizes as a function for 1-Wasserstein radius estimated by 200 independent simulation runs. (a)$N=20$, (b) $N=100$, (c) $N = 200$.}}
	\label{radius}
\end{figure*}

Experiments on different sample sizes are performed as well. The out-of-sample performance averaged over $200$ independent experiments is  presented in Figure \ref{fig:Sam}.  Theorem \ref{asy_con} is confirmed by the out-of-sample performance improvement with the growing sample size.

\begin{figure}[htbp]
	\centering
	\includegraphics[scale=0.35]{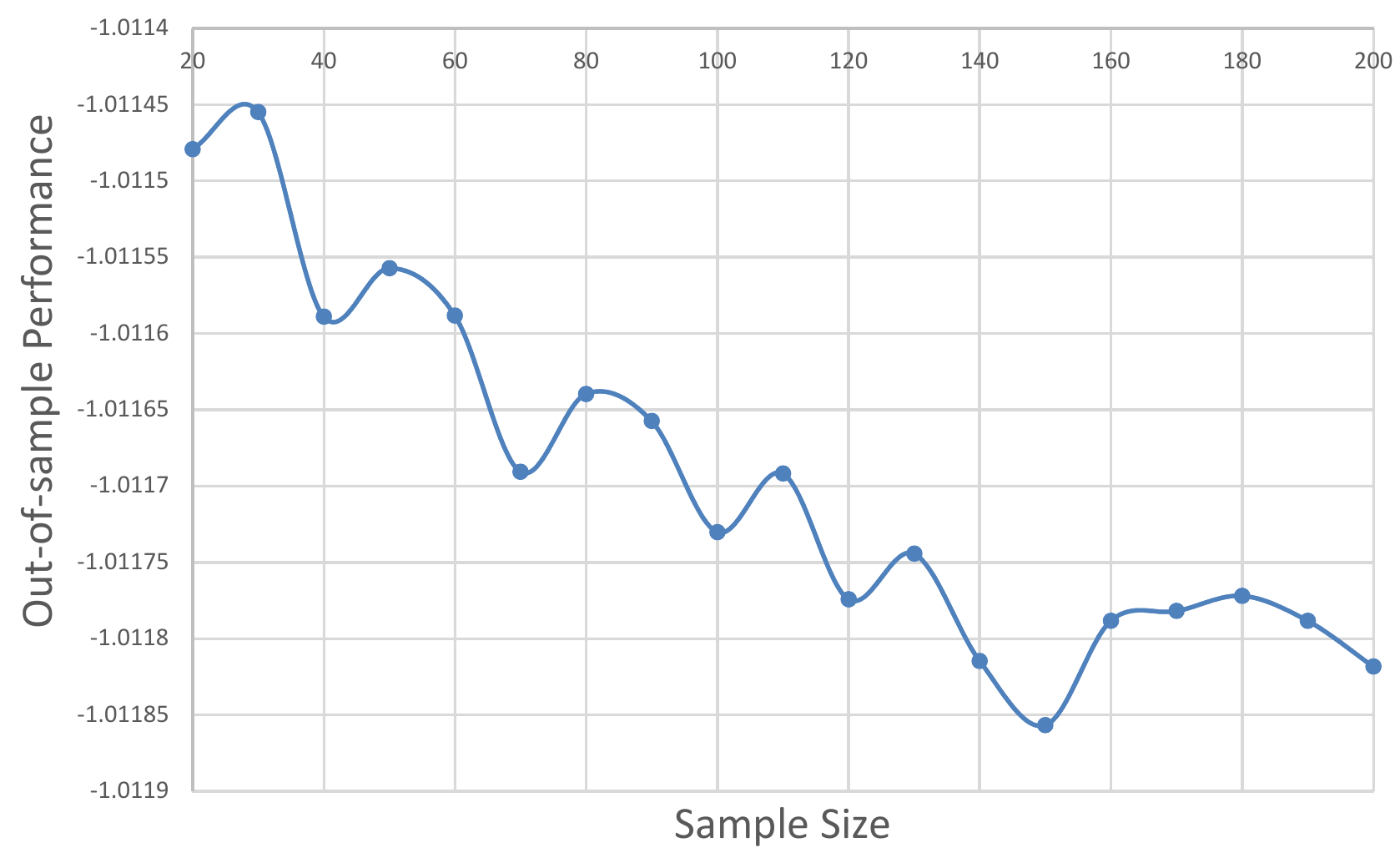}
	\caption{The averaged out-of-sample performance as a function of sample size $N$ for $200$ independent experiments.}
	\label{fig:Sam}
\end{figure}


\subsubsection{Comparisons with the State-of-the-art Methods}
In this subsection, we compare the proposed 1-Wasserstein DR model (denoted as DRW) with the SAA method and the DR model with the moment-based ambiguity set (denoted as DRM), where the first- and second-order uncertainty are borrowed from \cite{ling2017robust}. Let $N = \{20,30,50,100,200,300\}$. Due to the dependence of the radius $\epsilon_N$ on the sample dataset size, we tune it to ensure a good out-of-sample performance.

We adopt the percentage difference
$$\left(\frac{\text{DR}}{\text{SAA}}-1\right)\times 100\%$$
to compare the out-of-sample performance of those models, where DR denotes the out-of-sample performance of the DR two-stage problem and SAA denotes that of the SAA method.

\renewcommand{\arraystretch}{1} 
\begin{table}[htb]
	\centering
	\caption{Percentage differences of out-of-sample performance(in $\%$) between the DR models and the SAA }
	\begin{tabular*}{0.48\textwidth} {@{}@{\extracolsep{\fill}}ccccccc@{}}
		\toprule[1 pt]
		$N$    & 20 & 30 & 50 & 100  & 200  & 300   \\ \midrule
		DRW & 1.1 &1.6& 1.7 & 2.1 & 4.1 & 4.8  \\ 	
		DRM & -1.3 & -0.7 & 0.7 & 1.5 & 3.6 & 3.5  \\
		\bottomrule[1 pt]
	\end{tabular*}
	\label{com_out}	
\end{table}

\renewcommand{\arraystretch}{1} 
\begin{table}[htb]
	\centering
	\caption{Averaged computation time (second) of different methods}
	\begin{tabu} to 0.48\textwidth {X[1,c] X[1,c]  X[1,c]X[1,c] X[1,c]  X[1,c] X[1,c]}
		\toprule[1 pt]
		$N$    & 20 & 30 & 50 & 100  & 200  & 300   \\ \midrule
		DRW & 0.14 & 0.15& 0.15 & 0.17 & 0.16 & 0.19 \\ 	
		DRM & 0.12 & 0.14 & 0.14 & 0.16 & 0.15 & 0.16  \\
		SAA & 0.13 & 0.15 & 0.16 & 0.17 & 0.16 & 0.16  \\
		\bottomrule[1 pt]		
	\end{tabu}
	\label{com_tim}	
\end{table}

Comparisons in terms of the out-of-sample performance and computation time
are presented in Table \ref{com_out} and Table \ref{com_tim} respectively. A positive value in Table \ref{com_out} implies a better performance of the DR method than the SAA. Table \ref{com_out} indicates the best out-of-sample performance of our proposed method among all models. Importantly, it can also be solved in an acceptable time even under a large sample dataset.

\subsection{The Two-stage Material Order Problem}
Algorithm \ref{algo_whole} is applied to solve the DR two-stage ordering problem in Example \ref{exmp-c}. We omit the comparison with the moment-based model  since there is no effective method to solve it \cite{ling2017robust}.

\subsubsection{Problem Specification}
Consider the crude oil order problem for the gasoline and fuel oil supply stated in \cite{kall1994stochastic}). The oil is from two countries and can be viewed as different materials. Then the coefficients of the material order problem in  Example \ref{exmp-c} is set as
\begin{flalign}
	 &\bm{c}=[2,3]^T, \bm{d}=[7,12]^T, u = 100, \  \nonumber \\
	& A(\bm{\xi})=\left[
	\begin{matrix}
		2 + \xi_1 & 3 \\
		6 & 3.4 + \xi_2  \\
	\end{matrix}
	\right],
	\bm{b}(\bm{\xi})=\left[
	\begin{matrix}
		180 + \xi_3  \\
		162 + \xi_4  \\
	\end{matrix}
	\right],& \nonumber
\end{flalign}
where $\bm{\xi} \in \mathbb{R}^4$ is a random vector with an unknown distribution and the recourse matrix $B$ is the identity matrix. We assume that $\bm{\xi}$ follows a Gaussian distribution $\mathcal{N}(\bm{\mu},\bm{\Sigma})$ with $ \bm{\mu} =[0,0,0,0]^T$ and $\bm{\Sigma} = \text{Diag}([9,12,0.21,0.16]^T)$, and generate $N$ samples to construct the 1-Wasserstein ball $\mathcal{F}_N$.

\subsubsection{Test the Tightness of Bounds}

We test the tightness of the proposed bounds in MP and SuP for an optimal function value (O.F.V) and the first-stage cost over the 1-Wasserstein ball with different radii $\epsilon_N$. Obviously, the extreme points of the set $\mathcal{P}=\{\bm{p} \ge 0: \bm{p}\le \bm{d}\} = \{\bm{p}\in \mathbb{R}^2_+: p_1 \le 7,p_2 \le 12 \}$ are $[0,0]^T, [0,12]^T, [7,0]^T \  \text{and} \ [7,12]^T$. Hence, we can solve \eqref{primal} directly with explicitly known extreme points and compare with Algorithm \ref{algo_whole}. Let $(x^d_1,x^d_2)$ denote the solution obtained via solving \eqref{primal} directly and $(x^{a}_1,x^{a}_2)$ obtained by Algorithm \ref{algo_whole}.  Table \ref{tab2} indicates that the two methods under different 1-Wasserstein radius obtain identical results.

\renewcommand{\arraystretch}{1.1} 
\begin{table*}[htb]
	\centering
	\caption{The optimal solutions under different methods with different 1-Wasserstein ball radii $\epsilon_N$ when sample size $N = 500$}
	\begin{tabu} to 0.8\textwidth{X[1,c] X[1,c]  X[1,c]  X[1,c] X[1,c]  X[1,c] X[1,c]}
		\toprule[1 pt]
		$\epsilon_N$ & \ \ 0.01   \ \ &  0.21 &  0.41 \ \ &  0.61 \ \ &  0.81 \ \ & 1\\ \hline
		$(x^d_1,x^d_2)$ & (42.7,57.2) & (41.2,50.8) &(38.7,41.5) &(36.2,32.4)& (34.7,26.4) &(33.4,22.5) \\
		$(x^{a}_1,x^{a}_2)$ & (42.7,57.2) & (41.2,50.8) &(38.7,41.5)&(36.2,32.4)& (34.7,26.4) &(33.4,22.5) \\
		\bottomrule[1 pt]	
	\end{tabu}
	\label{tab2}	
\end{table*}

The O.F.V. and the first-stage cost compared to that of the method with known extreme points under $500$ samples is shown in Fig.\ref{bound_per_OFV} and Fig.\ref{bound_per_first}. We observe that both the lower bound and upper bound are tight, regardless of the radius of the 1-Wasserstein ball. Thus, these bounds can be viewed as a good reference to verify the performance of our algorithm.

\begin{figure}[htb]
\centering
    \subfigure[ ]{
		\label{bound_per_OFV}
		\includegraphics[width=0.35 \textwidth]{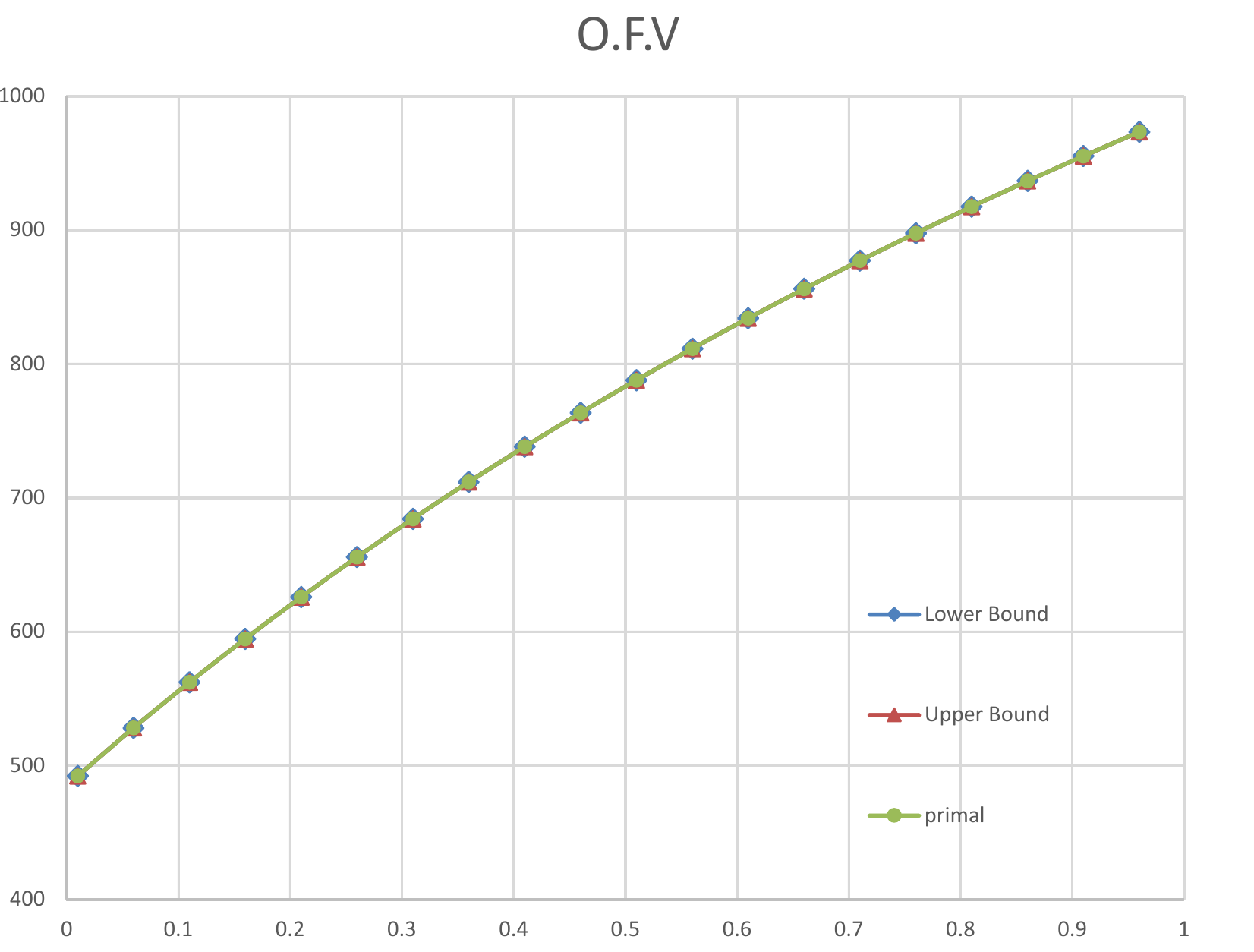}
	}
    \subfigure[ ]{
		\label{bound_per_first}
		\includegraphics[width=0.35 \textwidth]{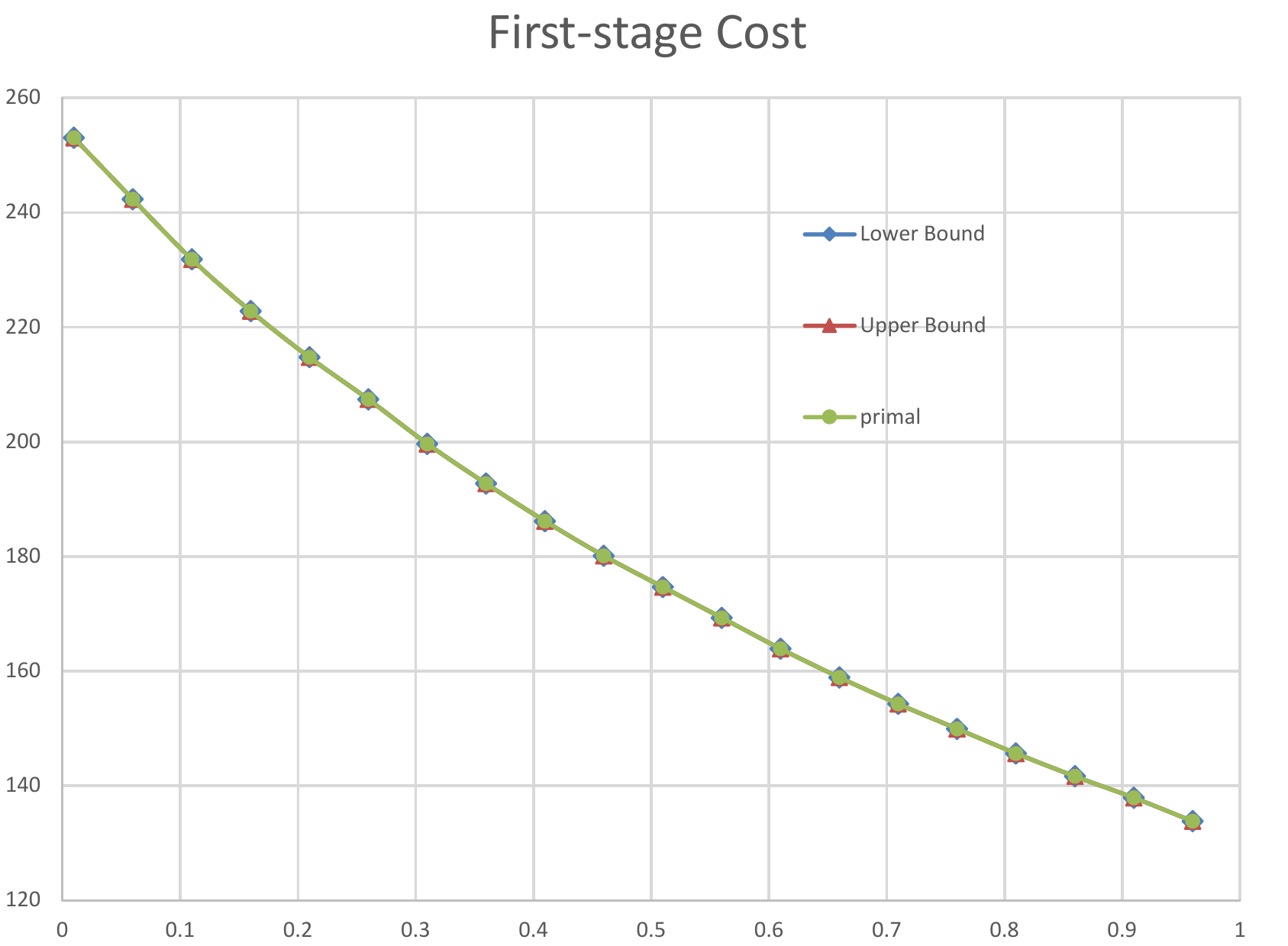}
	}
	\caption{The averaged performance of the proposed bounds for O.F.V. and the first-stage cost under the 1-Wasserstein ball with different radii. (a) O.F.V (b) the first-stage cost}
\end{figure}


\renewcommand\arraystretch{0.3}
\begin{table}[htb]
	\centering
	\caption{The averaged  number of extreme points under different sample sizes }
	\begin{tabu} to 0.48\textwidth{X[1,c] X[1,c] X[1,c] X[1,c]X[1,c] X[1,c] X[1,c] X[1,c] X[1,c] X[1,c]}
		\toprule[1 pt]
		$N$   & 10  &  20   & 30   &  50   & 100  &   200  &  300   & 500  & 1000 \\ \midrule
		Num & 3.68 & 3.74 & 3.98 & 3.96 & 4 & 4 & 4 & 4 & 4 \\ 		
		\bottomrule[1 pt]	
	\end{tabu}
	\label{tab4}	
\end{table}

\renewcommand\arraystretch{0.3}
\begin{table}[!htb]
	\centering
	\caption{The averaged number of iterations  under different sample sizes}
	\begin{tabu} to 0.48\textwidth{X[1,c] X[1,c] X[1,c] X[1,c]X[1,c] X[1,c] X[1,c] X[1,c] X[1,c] X[1,c]}
		\toprule[1 pt]
		$N$   & 10  &  20   & 30   &  50   & 100  &   200  &  300   & 500  & 1000 \\ \midrule
		Ite & 3.78 & 3.84 & 3.94 & 3.94 & 4 & 4 & 4 & 4 & 4 \\
		\bottomrule[1 pt]	
	\end{tabu}
	\label{ite-num}	
\end{table}

\begin{figure}[!htb]
	\centering
	\includegraphics[scale=0.35]{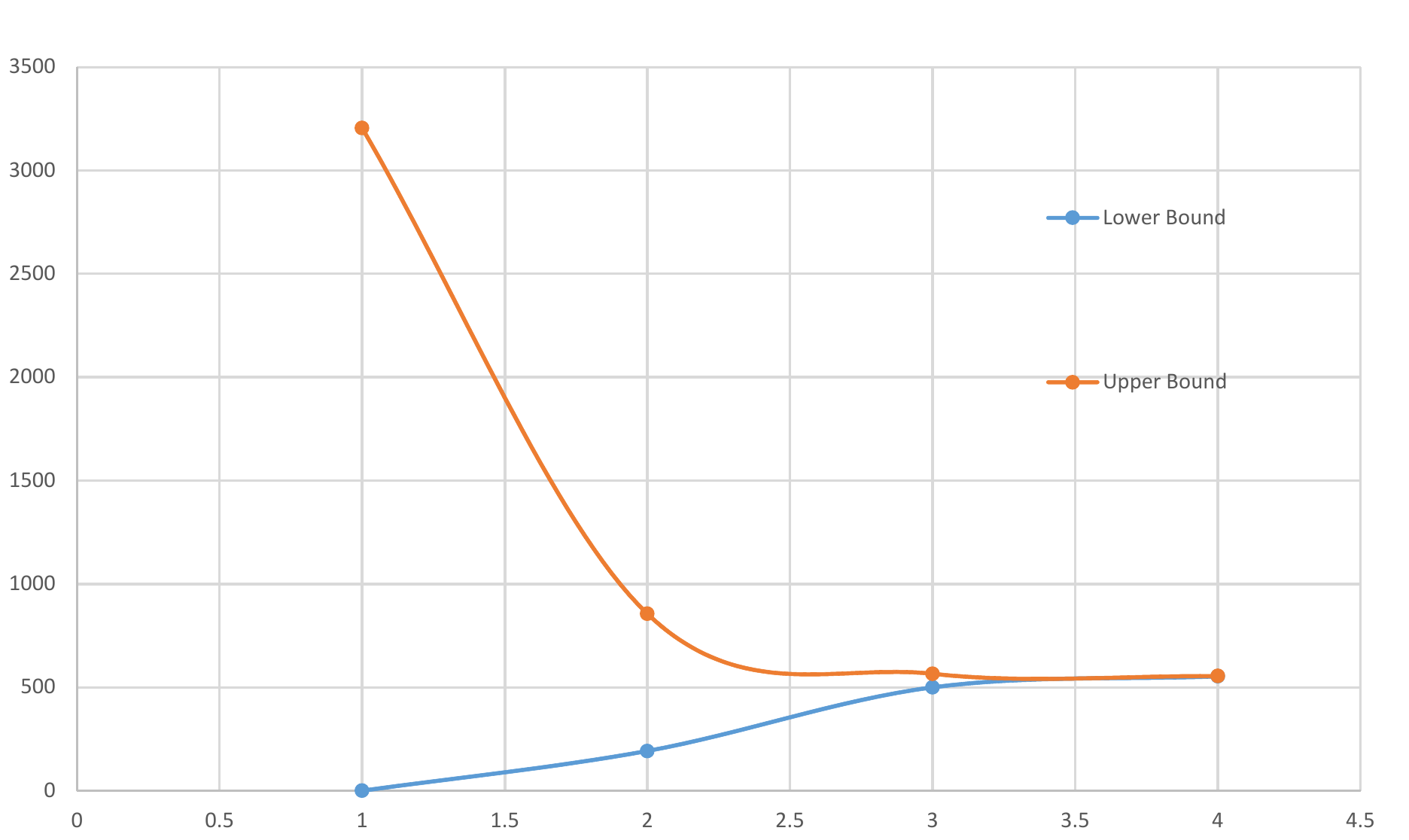}
	\caption{The convergence of the O.F.V for the two-stage program with $500$ samples.}
	\label{tendency_1}
\end{figure}

Fig.\ref{tendency_1} shows the tendency of the upper and lower bound for the proposed two-stage program in a single experiment. We record the averaged number of the extreme points and  iterations in Algorithm \ref{algo_whole} under different sample sizes over $100$ independent experiments in Table \ref{tab4} and Table \ref{ite-num}, both of which validate the effectiveness of Algorithm \ref{algo_whole}.



\subsubsection{The Test for High Dimension}
A direct enumeration of all extreme points of the polyhedron $\mathcal{P} = \{\bm{p}\in \mathbb{R}^M_+:B^T\bm{p}\le\bm{d}\}$ with a large $M$ is computational demanding \cite{khachiyan2009generating}. In this subsection, we consider a high dimension problem to verify the efficiency of Algorithm \ref{algo_whole}, i.e.,
\begin{flalign*}
& u = 1000, ~\bm{x} \in \mathbb{R}^{20}, ~A(\bm \xi) \in \mathbb{R}^{20\times 20}, ~ \bm{b}(\bm \xi)\in \mathbb{R}^{20}, \\
& \bm{c} = [2,3,1,4,5,2,4,3,4,2,5,4,4,2,6,2,4,3,1,2]^T, \\
& \bm{d} = [7,9,4,6,8,5,6,8,10,7,12,10,6,7,9,5,11,10,5,8]^T,
\end{flalign*}
where $A(\bm \xi)$ and $\bm{b}(\bm \xi)$ are affinely dependent on the random vector $\bm \xi$ and $B$ is the identity matrix.


Fig.\ref{bound_hper_OFV} and Fig.\ref{bound_hper_first} report the averaged performance of our proposed bounds for the O.F.V and the first-stage cost under different 1-Wasserstein radii $\epsilon_N$ when the sample size $N = 500$. As  previous subsection, these proposed bounds are tight as well.

\begin{figure}[htb]
\centering
    \subfigure[ ]{
		\label{bound_hper_OFV}
		\includegraphics[width=0.35 \textwidth]{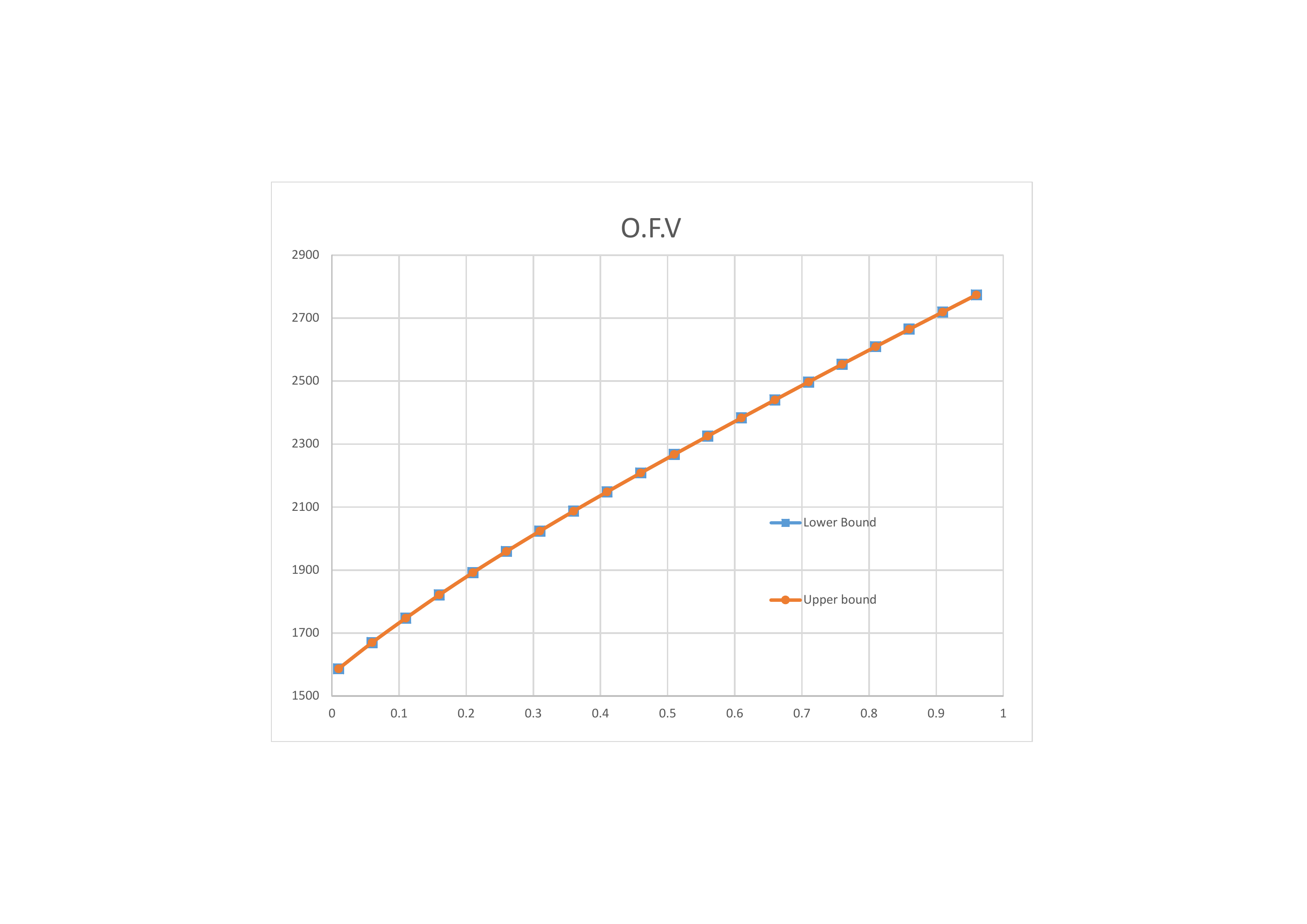}
	}
    \subfigure[ ]{
		\label{bound_hper_first}
		\includegraphics[width=0.35 \textwidth]{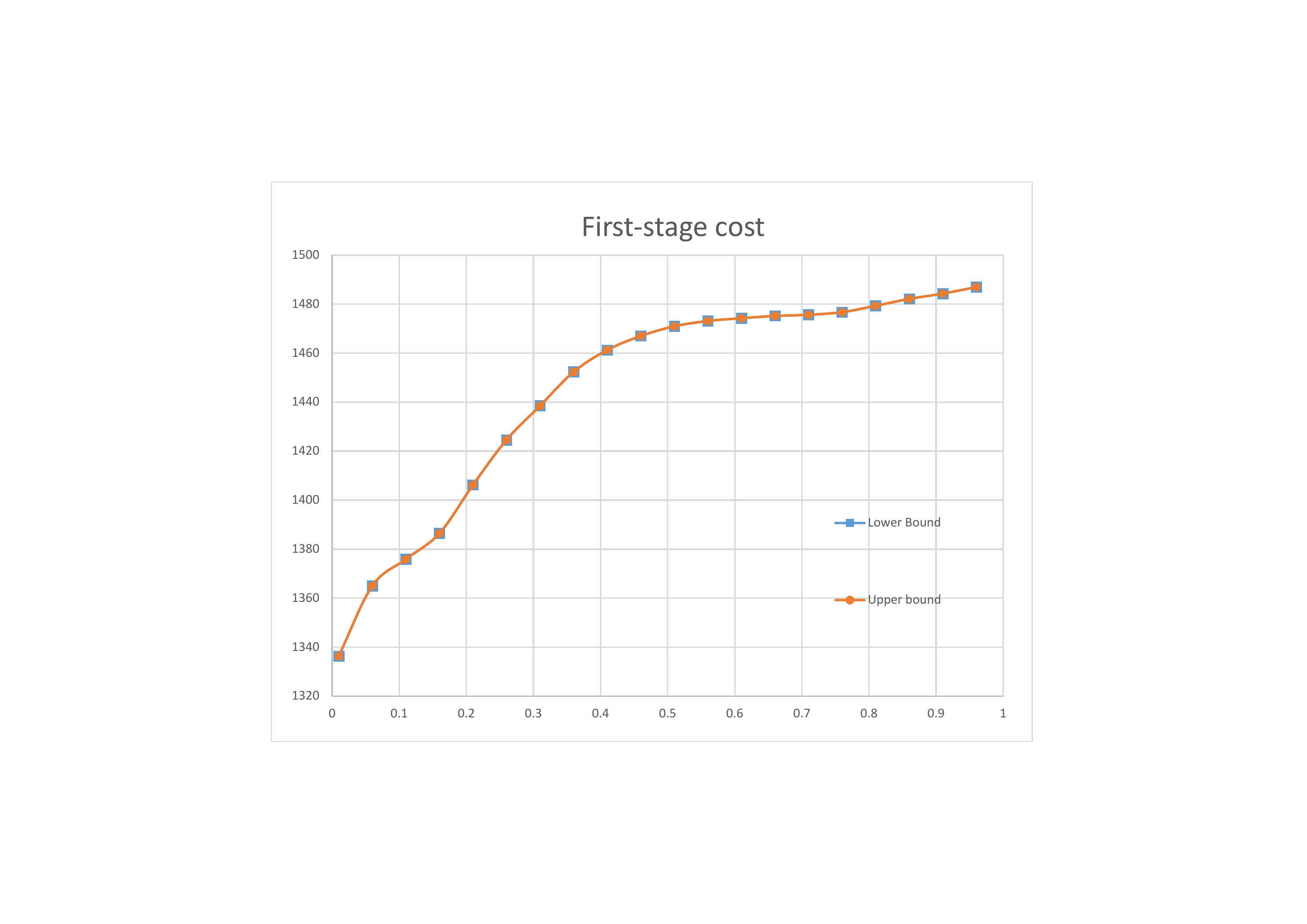}
	}
	\caption{The averaged performance of the proposed bounds for O.F.V. and first-stage cost under the 1-Wasserstein ball with different radii. (a) O.F.V (b) first-stage cost}
\end{figure}

We record the averaged computation time, the number of  extreme points  and  iterations in Algorithm \ref{algo_whole} over $100$ independent simulations as sample size $N$ varies from $10$ to $1000$ in Table \ref{tab5}, Table \ref{tab6} and Table \ref{tab7} respectively. The convergence of the proposed algorithm in a single experiment is also illustrated in
Fig.\ref{tendency}.

\renewcommand{\arraystretch}{0.3} 
\begin{table}[h]
	\centering
	\caption{The averaged computation time (second) under different sample sizes }
	\begin{tabu} to 0.48\textwidth{X[1,c] X[1,c] X[1,c] X[1,c] X[1,c] X[1,c] X[1,c] X[1,c] X[1,c] X[1,c]}
		\toprule[1 pt]
		$N$   & 10  &  20   & 30   &  50   & 100  &   200  &  300   & 500  & 1000 \\ \midrule
		Time & 10.9 & 11.6 & 11.6 & 11.6 & 12.3 & 13.9 & 17.2 & 23.3 & 36.2 \\
		\bottomrule[1 pt]
	\end{tabu}
	\label{tab5}	
\end{table}

\renewcommand{\arraystretch}{0.3} 
\begin{table}[htb]
	\centering
	\caption{The averaged number of extreme points  under different sample sizes}
	\begin{tabu} to 0.48\textwidth{X[1,c] X[1,c] X[1,c] X[1,c] X[1,c] X[1,c] X[1.2,c] X[1.2,c] X[1.2,c] X[1.2,c]}
		\toprule[1 pt]
		$N$   & 10  &  20   & 30   &  50   & 100  &   200  &  300   & 500  & 1000 \\ \midrule	
		Num & 35.2 & 46.5 & 49.2 & 60.3 & 72.4 & 100.1 & 123.7 & 156.4 & 181.1 \\
		\bottomrule[1 pt]
	\end{tabu}
	\label{tab6}	
\end{table}

\renewcommand{\arraystretch}{0.3} 
\begin{table}[!htb]
	\centering
	\caption{The averaged number of iterations  under different sample sizes }
	\begin{tabu} to 0.48\textwidth{X[1,c] X[1,c] X[1,c] X[1,c] X[1,c] X[1,c] X[1,c] X[1,c] X[1,c] X[1,c]}
		\toprule[1 pt]
		$N$   & 10  &  20   & 30   &  50   & 100  &   200  &  300   & 500  & 1000 \\ \midrule	
		Ite & 10.28 & 9.58 & 9.32 & 8.92 & 8.80 & 8.54 & 8.58 & 8.16 & 8.46\\
		\bottomrule[1 pt]
	\end{tabu}
	\label{tab7}	
\end{table}

\begin{figure}[!htb]
	\centering
	\includegraphics[scale=0.35]{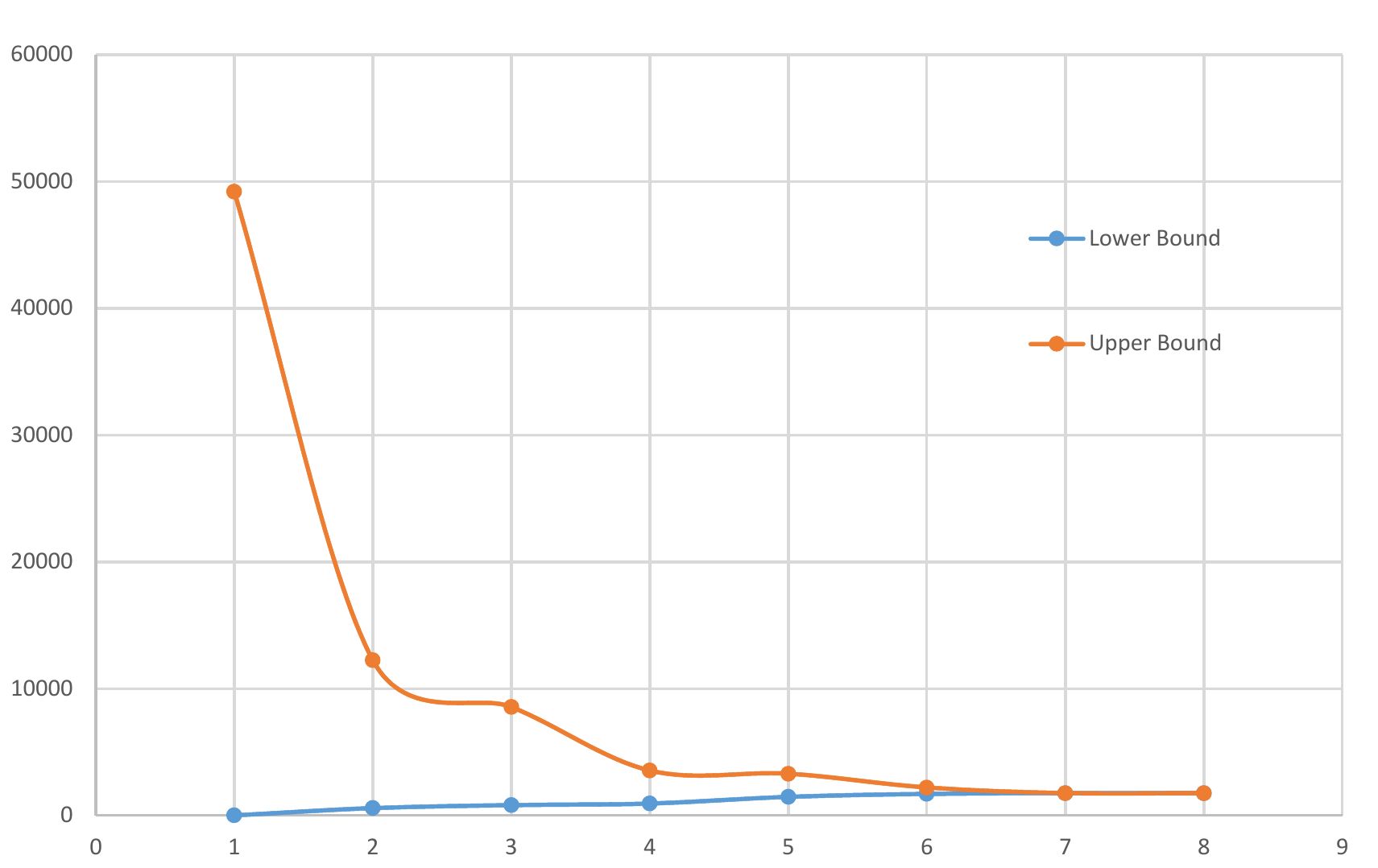}
	\caption{The convergence of the O.F.V for the two-stage program with $500$ samples.}
	\label{tendency}
\end{figure}

Results show that Algorithm \ref{algo_whole} converges in a reasonable time even for the problem in a high dimension under a large sample dataset.
The number of extreme points required in our algorithm is far smaller than the total number of extreme points.


\section{Conclusion} \label{con}
We have proposed a novel SOCP approach to solve the data-driven DR two-stage linear programs over 1-Wasserstein balls. The model with distribution uncertainty in the objective function is reformulated as a solvable SOCP problem.
While the DR model over the moment-based ambiguity set is generally unsolvable, we propose a constraint generation algorithm with provable convergence to approximately solve the NP-hard model with distribution uncertainty only in constraints. We explicitly derive a distribution achieving the worst-case cost. Numerical results validate the good out-of-sample performance for our model and the high efficiency of the proposed algorithm.

\bibliographystyle{IEEEtran}
\bibliography{mybib}

\begin{thebibliography}{10}
\providecommand{\url}[1]{#1}
\csname url@samestyle\endcsname
\providecommand{\newblock}{\relax}
\providecommand{\bibinfo}[2]{#2}
\providecommand{\BIBentrySTDinterwordspacing}{\spaceskip=0pt\relax}
\providecommand{\BIBentryALTinterwordstretchfactor}{4}
\providecommand{\BIBentryALTinterwordspacing}{\spaceskip=\fontdimen2\font plus
\BIBentryALTinterwordstretchfactor\fontdimen3\font minus
  \fontdimen4\font\relax}
\providecommand{\BIBforeignlanguage}[2]{{%
\expandafter\ifx\csname l@#1\endcsname\relax
\typeout{** WARNING: IEEEtran.bst: No hyphenation pattern has been}%
\typeout{** loaded for the language `#1'. Using the pattern for}%
\typeout{** the default language instead.}%
\else
\language=\csname l@#1\endcsname
\fi
#2}}
\providecommand{\BIBdecl}{\relax}
\BIBdecl

\bibitem{zhang2018ambulance}
R.~Zhang and B.~Zeng, ``Ambulance deployment with relocation through robust
  optimization,'' \emph{IEEE Transactions on Automation Science and
  Engineering}, vol.~16, no.~1, pp. 138--147, 2018.

\bibitem{9046836}
F.~{Wang} and F.~{Ju}, ``Transient and steady-state analysis of multistage
  production lines with residence time limits,'' \emph{IEEE Transactions on
  Automation Science and Engineering}, pp. 1--13, 2020.

\bibitem{paridari2015robust}
K.~Paridari, A.~Parisio, H.~Sandberg, and K.~H. Johansson, ``Robust scheduling
  of smart appliances in active apartments with user behavior uncertainty,''
  \emph{IEEE Transactions on Automation Science and Engineering}, vol.~13,
  no.~1, pp. 247--259, 2015.

\bibitem{9093973}
S.~M. {Hosseini}, R.~{Carli}, and M.~{Dotoli}, ``Robust optimal energy
  management of a residential microgrid under uncertainties on demand and
  renewable power generation,'' \emph{IEEE Transactions on Automation Science
  and Engineering}, pp. 1--20, 2020.

\bibitem{ben2009robust}
A.~Ben-Tal, L.~El~Ghaoui, and A.~Nemirovski, \emph{Robust optimization}.\hskip
  1em plus 0.5em minus 0.4em\relax Princeton University Press, 2009, vol.~28.

\bibitem{zhao2018inventory}
Y.~Zhao, X.~Xu, and H.~Li, ``Inventory-constrained throughput optimization for
  stochastic customer orders,'' \emph{IEEE Transactions on Automation Science
  and Engineering}, 2018.

\bibitem{shapiro2009lectures}
A.~Shapiro, D.~Dentcheva, and A.~Ruszczy{\'n}ski, \emph{Lectures on stochastic
  programming: modeling and theory}.\hskip 1em plus 0.5em minus 0.4em\relax
  SIAM, 2009.

\bibitem{shapiro1998simulation}
A.~Shapiro and T.~Homem-de Mello, ``A simulation-based approach to two-stage
  stochastic programming with recourse,'' \emph{Mathematical Programming},
  vol.~81, no.~3, pp. 301--325, 1998.

\bibitem{shapiro2002minimax}
A.~Shapiro and A.~Kleywegt, ``Minimax analysis of stochastic problems,''
  \emph{Optimization Methods and Software}, vol.~17, no.~3, pp. 523--542, 2002.

\bibitem{shi2018distributionally}
Z.~Shi, H.~Liang, S.~Huang, and V.~Dinavahi, ``Distributionally robust
  chance-constrained energy management for islanded microgrids,'' \emph{IEEE
  Transactions on Smart Grid}, vol.~10, no.~2, pp. 2234--2244, 2018.

\bibitem{van2015distributionally}
B.~P. Van~Parys, D.~Kuhn, P.~J. Goulart, and M.~Morari, ``Distributionally
  robust control of constrained stochastic systems,'' \emph{IEEE Transactions
  on Automatic Control}, vol.~61, no.~2, pp. 430--442, 2015.

\bibitem{Zhang2017Robust}
Y.~Zhang, S.~Song, Z.~J.~M. Shen, and C.~Wu, ``Robust shortest path problem
  with distributional uncertainty,'' \emph{IEEE Transactions on Intelligent
  Transportation Systems}, vol.~PP, no.~99, pp. 1--11, 2017.

\bibitem{xiong2016distributionally}
P.~Xiong, P.~Jirutitijaroen, and C.~Singh, ``A distributionally robust
  optimization model for unit commitment considering uncertain wind power
  generation,'' \emph{IEEE Transactions on Power Systems}, vol.~32, no.~1, pp.
  39--49, 2016.

\bibitem{bertsimas2010models}
D.~Bertsimas, X.~V. Doan, K.~Natarajan, and C.-P. Teo, ``Models for minimax
  stochastic linear optimization problems with risk aversion,''
  \emph{Mathematics of Operations Research}, vol.~35, no.~3, pp. 580--602,
  2010.

\bibitem{hanasusanto2016k}
G.~A. Hanasusanto, D.~Kuhn, and W.~Wiesemann, ``K-adaptability in two-stage
  distributionally robust binary programming,'' \emph{Operations Research
  Letters}, vol.~44, no.~1, pp. 6--11, 2016.

\bibitem{ling2017robust}
A.~Ling, J.~Sun, N.~Xiu, and X.~Yang, ``Robust two-stage stochastic linear
  optimization with risk aversion,'' \emph{European Journal of Operational
  Research}, vol. 256, no.~1, pp. 215--229, 2017.

\bibitem{chen2018distributionally}
Y.~Chen, Q.~Guo, H.~Sun, Z.~Li, W.~Wu, and Z.~Li, ``A distributionally robust
  optimization model for unit commitment based on kullback--leibler
  divergence,'' \emph{IEEE Transactions on Power Systems}, vol.~33, no.~5, pp.
  5147--5160, 2018.

\bibitem{jiang2018risk}
R.~Jiang and Y.~Guan, ``Risk-averse two-stage stochastic program with
  distributional ambiguity,'' \emph{Operations Research}, vol.~66, no.~5, pp.
  1390--1405, 2018.

\bibitem{esfahani2018data}
P.~M. Esfahani and D.~Kuhn, ``Data-driven distributionally robust optimization
  using the {W}asserstein metric: Performance guarantees and tractable
  reformulations,'' \emph{Mathematical Programming}, vol. 171, no. 1-2, pp.
  115--166, 2018.

\bibitem{wang2020solving}
Z.~Wang, K.~You, S.~Song, and Y.~Zhang, ``Solving {W}asserstein robust
  two-stage stochastic linear programs via second-order conic programming,'' in
  \emph{59th Conference on Decision and Control (Submitted)}.\hskip 1em plus
  0.5em minus 0.4em\relax IEEE, 2020.

\bibitem{zeng2013solving}
B.~Zeng and L.~Zhao, ``Solving two-stage robust optimization problems using a
  column-and-constraint generation method,'' \emph{Operations Research
  Letters}, vol.~41, no.~5, pp. 457--461, 2013.

\bibitem{ding2015parallel}
J.-Y. Ding, S.~Song, R.~Zhang, R.~Chiong, and C.~Wu, ``Parallel machine
  scheduling under time-of-use electricity prices: New models and optimization
  approaches,'' \emph{IEEE Transactions on Automation Science and Engineering},
  vol.~13, no.~2, pp. 1138--1154, 2015.

\bibitem{hanasusanto2018conic}
G.~A. Hanasusanto and D.~Kuhn, ``Conic programming reformulations of two-stage
  distributionally robust linear programs over {W}asserstein balls,''
  \emph{Operations Research}, vol.~66, no.~3, pp. 849--869, 2018.

\bibitem{xie2019tractable}
W.~Xie, ``Tractable reformulations of distributionally robust two-stage
  stochastic programs with $\infty-$ {W}asserstein distance,'' \emph{arXiv
  preprint arXiv:1908.08454}, 2019.

\bibitem{bertsimas2018adaptive}
D.~Bertsimas, M.~Sim, and M.~Zhang, ``Adaptive distributionally robust
  optimization,'' \emph{Management Science}, vol.~65, no.~2, pp. 604--618,
  2018.

\bibitem{birge2011introduction}
J.~R. Birge and F.~Louveaux, \emph{Introduction to stochastic
  programming}.\hskip 1em plus 0.5em minus 0.4em\relax Springer Science \&
  Business Media, 2011.

\bibitem{kall1994stochastic}
P.~Kall, S.~W. Wallace, and P.~Kall, \emph{Stochastic {P}rogramming}.\hskip 1em
  plus 0.5em minus 0.4em\relax Springer, 1994.

\bibitem{cantelli1933sulla}
F.~P. Cantelli, ``Sulla determinazione empirica delle leggi di probabilita
  ({O}n the empirical determination of the laws of probability),'' \emph{Giorn.
  Ist. Ital. Attuari}, vol.~4, pp. 421--424, 1933.

\bibitem{ambrosio2013user}
L.~Ambrosio and N.~Gigli, ``A user's guide to optimal transport,'' in
  \emph{Modelling and optimisation of flows on networks}.\hskip 1em plus 0.5em
  minus 0.4em\relax Springer, 2013, pp. 1--155.

\bibitem{Shapiro2001On}
A.~Shapiro, ``On duality theory of conic linear problems,'' \emph{Nonconvex
  Optimization and Its Applications}, vol.~57, pp. 135--165, 2001.

\bibitem{wang2020wasserstein}
Z.~Wang, K.~You, S.~Song, and Y.~Zhang, ``Wasserstein distributionally robust
  shortest path problem,'' \emph{European Journal of Operational Research},
  2020.

\bibitem{bodlaender1990computational}
H.~L. Bodlaender, P.~Gritzmann, V.~Klee, and J.~Van~Leeuwen, ``Computational
  complexity of norm-maximization,'' \emph{Combinatorica}, vol.~10, no.~2, pp.
  203--225, 1990.

\bibitem{huang2016consensus}
K.~Huang and N.~D. Sidiropoulos, ``Consensus-admm for general quadratically
  constrained quadratic programming,'' \emph{IEEE Transactions on Signal
  Processing}, vol.~64, no.~20, pp. 5297--5310, 2016.

\bibitem{khachiyan2009generating}
L.~Khachiyan, E.~Boros, K.~Borys, V.~Gurvich, and K.~Elbassioni, ``Generating
  all vertices of a polyhedron is hard,'' in \emph{Twentieth Anniversary
  Volume:}.\hskip 1em plus 0.5em minus 0.4em\relax Springer, 2009, pp. 1--17.

\end{thebibliography}

\begin{IEEEbiography}[{\includegraphics[width=1in,height=1.25in,clip,keepaspectratio]{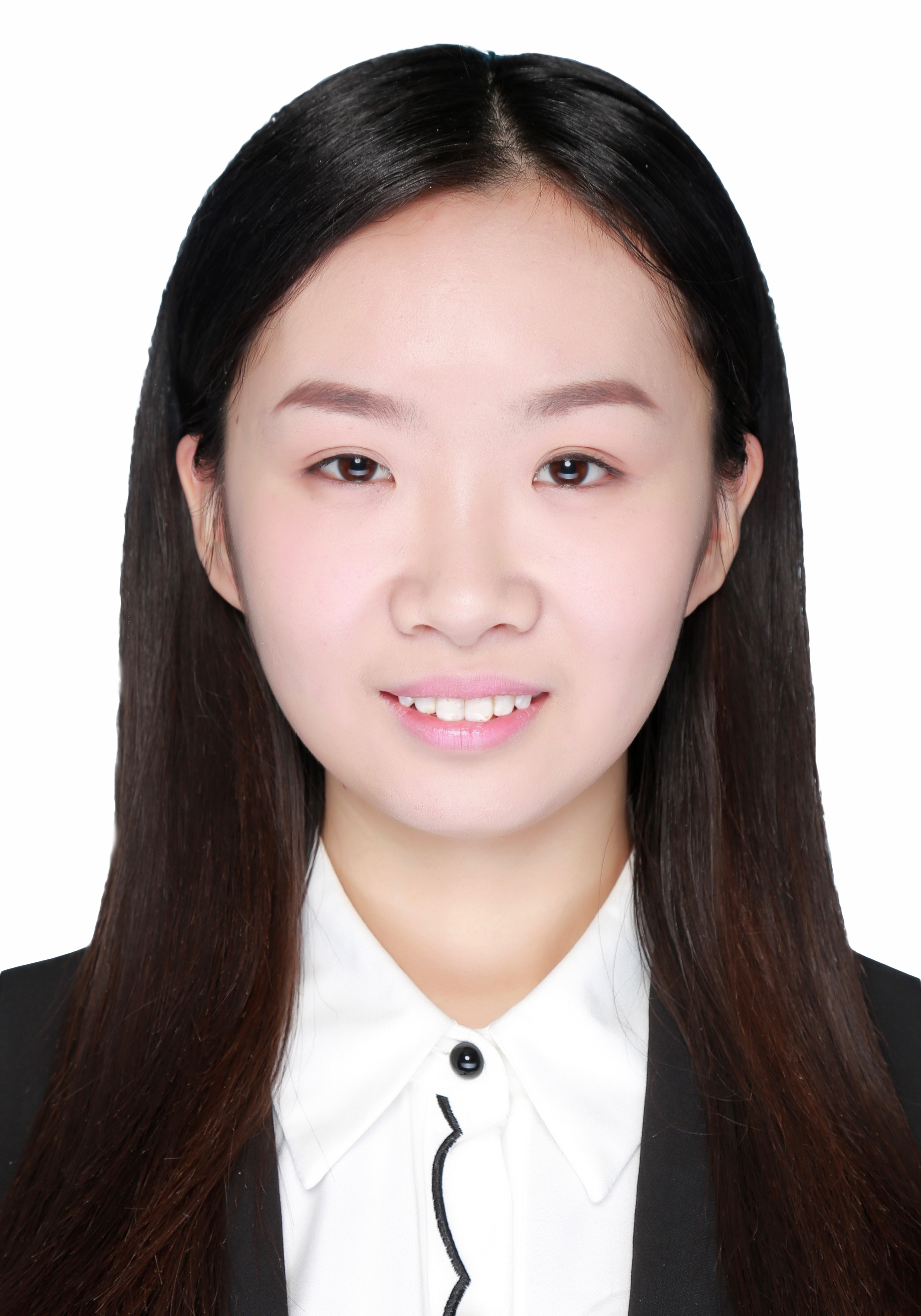}}]{Zhuolin Wang} received the B.S. degree from the School of Information Computer Science and Technology, Beijing Normal University, Beijing, China, in 2017. She is currently pursuing the Ph.D. degree at the Department of Automation, Tsinghua University, Beijing, China. Her  research interests include distributionally robust optimization, robust control, and their applications.
\end{IEEEbiography}

\begin{IEEEbiography}[{\includegraphics[width=1in,height=1.25in,clip,keepaspectratio]{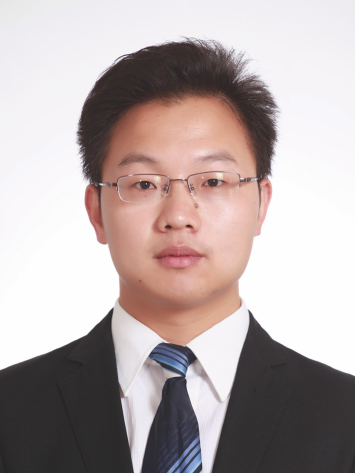}}]
{Keyou You} (SM'17)  received the B.S. degree in Statistical Science from Sun Yat-sen University, Guangzhou, China, in 2007 and the Ph.D. degree in Electrical and Electronic Engineering from Nanyang Technological University (NTU), Singapore, in 2012. After briefly working as a Research Fellow at NTU, he joined Tsinghua University in Beijing, China where he is now a tenured Associate Professor in the Department of Automation. He held visiting positions at Politecnico di Torino, The Hong Kong University of Science and Technology, The University of Melbourne and etc. His current research interests include networked control systems, distributed optimization and learning, and their applications.

	Dr. You received the Guan Zhaozhi award at the 29th Chinese Control Conference in 2010, the CSC-IBM China Faculty Award in 2014 and the ACA (Asian Control Association) Temasek Young Educator Award in 2019. He was selected to the National 1000-Youth Talent Program of China in 2014 and received the National Science Fund for Excellent Young Scholars in 2017.
	
\end{IEEEbiography}

\begin{IEEEbiography}[{\includegraphics[width=1in,height=1.25in,clip,keepaspectratio]{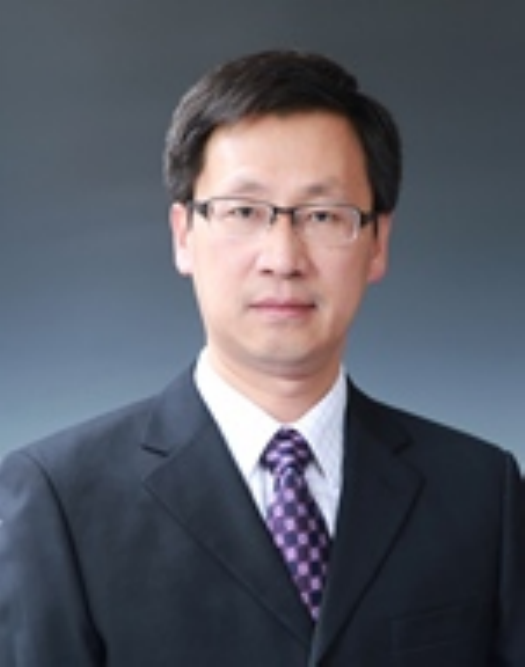}}]{Shiji Song}
received the Ph.D. degree in the Department of Mathematics from Harbin Institute of Technology in 1996. He is a professor in the Department of Automation, Tsinghua University. His research interests include system modeling, control and optimization, computational intelligence and pattern recognition.
\end{IEEEbiography}

\begin{IEEEbiography}[{\includegraphics[width=1in,height=1.25in,clip,keepaspectratio]{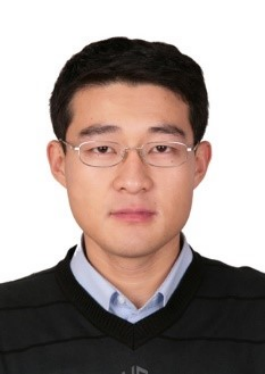}}]{Yuli Zhang}
 received the Ph.D. degree from Tsinghua University, Beijing, China, in 2014. He is an Associate Professor in the School of Management and Economics, Beijing Institute of Technology. His research interest includes robust optimization, intelligent transportation systems, and intelligent
power grids.
\end{IEEEbiography}

\end{document}